\documentclass[12pt]{article}
\usepackage{hyperref, color}
\usepackage{cmap}                        
\usepackage[utf8]{inputenc}            
\usepackage[russian,english]{babel}
\usepackage[left=1in,right=1in,top=1in,bottom=1.5in]{geometry} 
\usepackage{amssymb,amsmath, amsthm, amscd,ifthen}
\usepackage{graphicx}
\usepackage{caption}
\usepackage[T2A]{fontenc}
\usepackage{tikz}
\usetikzlibrary{patterns,decorations.text,positioning,arrows,shapes,decorations.markings}

\newcommand{\setbuilder}[2]{\left\{#1\ \colon #2\right\}}

\newcommand{\R}{{\mathbb R}}

\newtheorem*{thm*}{Theorem}

\newtheorem*{cla*}{Claim}

\newtheorem{thm}{Theorem}

\newtheorem{gypo}{Conjecture}

\newtheorem{lem}[thm]{Lemma}
\newtheorem{cla}[thm]{Claim}
\newtheorem{ques}{Question}

\date{}

\newtheorem{prop}[thm]{Proposition}

\newtheorem{prb}{Problem}
\newcommand{\ep}{\varepsilon}

\DeclareMathOperator{\eps}{\varepsilon}

\date{}
\title{Nearly $k$-distance sets}
\author{N\'ora Frankl\thanks{Alfréd Rényi Institute of Mathematics. Email: {\tt nfrankl@renyi.hu } Research was partially supported by the National Research, Development,
and Innovation Office, NKFIH Grant K119670 and ERC Advanced Grant "GeoScape."} \and Andrey Kupavskii\thanks{Moscow Institute of Physics and Technology, G-SCOP, Universite Grenoble-Alpes, CNRS, France. Email: {\tt kupavskii@yandex.ru} \ \ Research supported  by the Russian Foundation for Basic Research (grant
no.  18-01-00355) and the Council for the Support of Leading Scientific Schools of the President of the
Russian Federation (grant no.
НШ-2540.2020.1).}}
\begin{document}
\maketitle
\begin{abstract}
     We say that a set of points $S\subset \R^d$ is an $\ep$-nearly $k$-distance set if there exist $1\le t_1\le \ldots\le t_k,$ such that the distance between any two distinct points in $S$ falls into $[t_1,t_1+\ep]\cup\ldots\cup[t_k,t_k+\ep]$. In this paper, we study the quantity 
    \vspace{-0.2cm}
    \[M_k(d) = \lim_{\eps\to 0}\max\{|S|\ :\ S\text{ is an }\ep\text{-nearly $k$-distance set in }\R^d\}
    \vspace{-0.2cm}
    \] 
    and its relation to the classical quantity $m_k(d)$: the size of the largest $k$-distance set in $\R^d$.
    We obtain that $M_k(d) = m_k(d)$ for $k=2,3$, as well as for any fixed $k$, provided that $d$ is sufficiently large. The last result answers a question, proposed by Erd\H{o}s, Makai and Pach.

    We also address a closely related Tur\'an-type problem, studied by Erd\H os, Makai, Pach, and Spencer in the 90's: given $n$ points in $\R^d$, how many pairs of them form a distance that belongs to $[t_1,t_1+1]\cup\ldots\cup[t_k,t_k+1],$ where $t_1,\ldots, t_k$ are fixed and any two points in the set are at distance at least $1$ apart?
    We establish the connection between this quantity and a quantity closely related to $M_k(d-1)$, as well as obtain an exact answer for the same ranges $k,d$ as above.
\end{abstract}

\section{Introduction}
Let us start with some definitions. We call any point set that determines at most $k$ distances a {\it $k$-distance set}. We denote by \emph{$m_k(d)$ the cardinality of the largest $k$-distance set in $\mathbb R^d$}.
A {\it balanced complete $s$-partite graph} on $n$ vertices is a graph whose vertices are partitioned into $k$ groups of size $\lfloor n/s\rfloor$ or $\lceil n/s\rceil$, and in which two vertices are connected by an edge if and only if they belong to different groups.
We denote by $T(n,s)$ the number of edges in a balanced complete $s$-partite graph on $n$ vertices. A set $P\subseteq \mathbb{R}^d$ is {\it separated} if  $\|p_1-p_2\|\geq 1$ for any $p_1, p_2\in P$ with $p_1\neq p_2$. Let us formulate the main result of this paper.

\begin{thm}\label{thm000}
The following holds for $k\le 3$ and any $d$, and for any fixed $k$ and $d\ge d_0(k)$.
\begin{itemize}
    \item[(i)] There exists $\ep=\ep(k,d)>0$, such that for any sequences $t_1,\ldots,t_k\ge 1$ of distances the following is true. If $P\subseteq \mathbb{R}^d$ is a set such that for any $p_i,p_j\in P$ with $p_i\neq p_j$  we have $\|p_i-p_j\|\in [t_1,t_1+\ep]\cup\ldots\cup[t_k,t_k+\ep]$, then $P$ has size at most $m_k(d)$. (Stated differently, $M_k(d) = m_k(d)$.)
    \item[(ii)] There exists $n_0(k,d)>0$, such that for any sequence $t_1,\ldots,t_k>0$ of distances the following is true. If $P\subset \R^d$ is a separated set of size $n\geq n_0$, then the number of pairs of points in $P$ whose distance lies in $[t_1,t_1+1]\cup\ldots\cup[t_k,t_k+1]$ is at most $T(n,m_k(d-1))$. This bound is sharp. Moreover, the same holds with intervals of the form $[t_i,t_i+cn^{1/d}]$ for some $c=c_0(k,d)>0$. 
\end{itemize}
\end{thm}
Theorem~\ref{thm000} (ii) for $k=1$ was proved by Erd\H os, Makai, Pach, and Spencer \cite{EMPS} (see Theorem~\ref{emps2} below). For $k=2$,  it was shown by Erd\H os, Makai, and Pach \cite{EMP2} in a slightly weaker form (see Theorem~\ref{emp3} below). Although (i) and (ii) are strongly related, the ``max clique'' problem in (i) has not been addressed before.
Theorem~\ref{thm000} is a combination of Theorems~\ref{nearly k} and~\ref{thmsmallkexact} below.

In the long introductory section that follows, we tried to address several points:

First, 
we relate the study of ``nearly-equal distances'' to that of ``equal distances''. The history of the latter is summarised in the next subsection, and the relation between the two notions is developed in Section~\ref{sec14}, in which we give constructions of nearly $k$-distance sets and compare them with the known constructions of $k$-distance sets.  

Second, 
we relate the older Tur\'an-type problem on the number of ``nearly-equal distances'' to the proposed problem of determining the largest nearly $k$-distance set. We give the history of this Tur\'an-type problem in Section~\ref{sec12}, introduce the study of nearly $k$-distance sets in Section~\ref{sec13}, and establish the first result that relates the  two questions  in Section~\ref{sec14}. 

Third,
we introduce some of the more technical, but important, notions used in the proofs in Section~\ref{proofs}, and in particular $\alpha$-flat nearly $k$-distance sets, defined in Section~\ref{sec131}. It is via the notion of $\alpha$-flat nearly $k$-distance sets that we actually establish the link between the ``nearly $k$-distance set'' problem and the Tur\'an-type problem.

Our main results are presented in Section~\ref{sec15}, with their proofs in Section~\ref{proofs}.

\subsection{Equal distances}

In 1946, Erd\H os \cite{Erd46} asked the following two questions, which greatly influenced the course of discrete geometry. Take a set $X$ of $n$ points on the plane. \begin{ques}\label{qu1} What is the smallest number of distinct distances that $X$ can determine?\end{ques} \begin{ques}\label{qu2} What is the maximum number of equal distances that $X$ can determine?\end{ques} These questions have a rich history, and we refer the reader to the book of Brass, Moser and Pach \cite{BMP} and the references therein. In the recent years, the development of algebraic methods in discrete geometry lead to a breakthrough of Guth and Katz \cite{GK}, who showed that the quantity in the first question is $\Omega(n/\log n)$, which almost matches Erd\H os' upper bound $O(n/\log^{1/2} n)$. For $\R^d$ with $d\ge 3$ Solymosi and Vu \cite{SV} proved that the quantity in the first question is
$\Omega\big(n^{(2/d)(d+1)/(d+2)}\big)$, while Erd\H os proved the close and conjecturably sharp upper bound $O(n^{2/d})$.

Even though Questions~1 and~2 seem to have exactly the same flavour, (which is backed by, e.g., the fact that upper bounds in Question~2 imply lower bounds in Question~1), much less is known about Question~2. The best upper bound $O(n^{4/3})$ is due to Spencer, Szemer\'edi, and Trotter \cite{SST}, and the best lower bound is due to Erd\H os is $\Omega(n^{1+c/\log\log n})$ for some $c>0$. Interestingly, this problem becomes much simpler in dimensions $d\ge 4$: there are point sets that determine quadratically many unit distances. (Brass \cite{Brass} and Van Wamelen \cite{VW} determined the maximum number of unit distances exactly for $d=4$ and Swanepoel  \cite{Swan}  for even $d\ge 6$ and large $n$, respectively.)
Stated slightly differently, Question~1 
asks to determine $m_k(2)$. The bounds mentioned above give $\Omega(k\log^{1/2} k)\le m_k(2)\le O(k\log k)$. In 1947, Kelly \cite{Kel} showed that $m_2(2) = 5$. 
Larman, Rogers, and Seidel \cite{LRS} showed that $m_2(d)\le \frac 12(d+1)(d+4)$. Several years later, Bannai, Bannai, and Stanton \cite{BBS} and independently Blokhuis \cite{Blok} found the following better bound, which additionally works for any $k$:
\begin{equation}\label{eqbbs}
  m_k(d)\le \binom{d+k}{k}.
\end{equation}

There is the following very natural construction of a $k$-distance set in $\mathbb R^d$ if $k\leq d+1$: in $\mathbb R^{d+1}$, take all vectors in $\{0,1\}^{d+1}$ with exactly $k$ many $1$'s. Then they lie on a sphere in the hyperplane $\sum x_i=k$ and determine only $k+1$ distinct scalar products (and thus only $k$ distinct distances). This gives the lower bound
\begin{equation}\label{eqlb}
  m_k(d)\ge \binom{d+1}{k}.
\end{equation}
We will refer to \eqref{eqbbs} and \eqref{eqlb} many times in the proofs. Table \ref{small}, taken from a paper of Sz{\"o}ll{\H{o}}si and {\"O}sterg{\aa}rd \cite{szol}, summarises the best known lower bounds on (and in some cases exact values of) $m_k(d)$ for small values of $k$ and $d$. Note also that it is easy to see that for any $d$ we have $m_1(d)=d+1$ and for any $k$ we have $m_k(1) = k+1$.

\begin{table}
\begin{center}
  \begin{tabular}{ | c | r | r | r | r | r | r | r | r | }
    \hline
    $_k\diagdown^d$ & 2 & 3 & 4 & 5 & 6 & 7 & 8 \\ \hline
    2 & 5 & 6 & 10 & 16 & 27 & 29 & 45 \\ \hline
    3 & 7 & 12 & 16 & $\geq 24$ & $\geq 40$ & $\geq 65$ & $\geq 121$\\ \hline
    4 & 9 & 13 & $\geq 25$ & $\geq 41$ & $\geq 73$ & $\geq 127$ & $\geq 241$\\ \hline
    5 & 12 & $\geq 20$ & $\geq 35$ & $\geq 66$ & $\geq 112$ & $\geq 168$ & $\geq 252$\\ \hline
    6 & 13 & $\geq 21$ & $\geq 40$ & $\geq 96$ & $\geq 141$ & $\geq 281$ & $\geq 505$\\ \hline
  \end{tabular}
  \caption{Lower bounds on $m_k(d)$.}\label{small}
\end{center}
\end{table}

\subsection{Nearly equal distances} \label{sec12}
Most of the results on $k$-distance sets use proofs with an algebraic flavour and often use results on incidences of points and surfaces. In the 90s, Erd\H os, Makai, Pach, and Spencer \cite{EMPS} proposed the following variant of the problem, which is \emph{not} of this type: given a set of $n$ points on the plane, how many of the distances between the points could be {\it nearly equal}, that is, that fall into the interval $[t,t+1]$ for some $t$? To avoid trivialities, we only consider separated sets.

It turns out that the answer to this question is very different from the answer to Question~2: we can have as many as $n^2 $ nearly equal distances in a separated set of size $2n$. To see this, take $2n$ points of the form $(x_1,x_2)$, where $x_1\in \{0,n^2\}$ and $x_2\in \{1,\ldots,n\}$. Then any distance between points with different $x_1$-coordinates is between $n^2$ and $\sqrt{n^4+n^2}<n^2+1$.

This example turns out to be optimal. The following theorem was proved by Erd\H os, Makai, Pach, and Spencer \cite{EMPS}.
\begin{thm}[\cite{EMPS}, Theorem~3]\label{emps1}
Let $P$ be a separated set of $n$ points in the plane. If $n$ is sufficiently large, then, for every $t>0$, the number of pairs of points in $P$ whose distance lies in the interval $[t,t+1]$ is at most $\lfloor n^2/4\rfloor$. This bound can be attained for every $t\ge t(n)$, where $t(n)$ is a suitable function of $n$.
\end{thm}
They have studied two natural types of generalisations: one deals with higher  dimensions and the other with more intervals in which the distances may fall. 
\begin{thm}[\cite{EMPS}, Theorem~5]\label{emps2}
Let $P$ be a separated set of $n$ points in $\mathbb R^d$. If $n$ is sufficiently large, then for every $t>0$, the number of pairs of points in $P$ whose distance lies in the interval $[t,t+1]$ is at most $T(n,d)$. This bound can be attained for every $t\ge t_0(d,n)$, where $t_0(d,n)$ is a suitable function of $d$ and $n$.
\end{thm}
In the case of several intervals, Erd\H os, Makai, and Pach \cite{EMP} proved the following theorem.
\begin{thm}[\cite{EMP}]\label{emp}
Let $P$ be a separated set of $n$ points in the plane, and let $k$ be a positive integer. Then for any $\gamma>0$ and sufficiently large $n$ the following is true. For any $t_1,\ldots,t_k>0$, the number of pairs of points in $P$ whose distance lies in $[t_1,t_1+1]\cup\ldots\cup[t_k,t_k+1]$ is at most 
\[\frac {n^2}2\Big(1-\frac 1{k+1}+\gamma\Big).\]
This estimate is tight for every fixed $k$ and for some $t_1=t_1(k,n),\ldots,t_k=t_k(k,n)$.
\end{thm}
Actually, they have proved something stronger: the allowed intervals in Theorem \ref{emps1}, \ref{emp2} and \ref{emp} are of the form $[t_i,t_i+c\sqrt{n}]$, $[t_i,t_i+c(d)n^{1/d}]$ and $[t_i,t_i+c(k,\gamma)\sqrt{n}]$ respectively, where $c,c(d)$ and $c(k,\gamma)$ are constants (only depending on their arguments).

In
\cite{MPS}
Makai, Pach and Spencer surveyed the results on this topic and also stated the following theorem that was supposed to appear in a follow-up paper by Erd\H os, Makai and Pach. 

\begin{thm}[Theorem 2.4 in \cite{MPS}, stated without proof]\label{emp2} Let $d\ge 2$ be an integer and let $P$ be a separated set of $n$ points in $\mathbb R^d$. For any $\gamma>0$ and
sufficiently large $n$ the following is true. For any $t_1,t_2>0$ the number of pairs of points in $P$ whose distance lies in $[t_1,t_1+1]\cup[t_2,t_2+1]$ is at most
\[\frac {n^2}2\left (1-\frac 1{m_2(d-1)}+\gamma \right ).\]
This bound is asymptotically tight 
for some $t_1=t_1(n)$ and $t_2=t_2(n)$.
\end{thm}
The proof of this theorem was kept in the form of handwritten notes until recently, when Makai and Pach \cite{EMP2} placed on the arXiv a typed version of those notes (joint with Paul Erd\H os). There, they prove Theorem~\ref{emp2} in a stronger form.

\begin{thm}[\cite{EMP2}, Theorem~1]\label{emp3}
  The statement of Theorem~\ref{emp2} is true. Moreover, even with intervals of the form $[t_i,t_i+c_d n^{1/d}]$, where $c_d>0$ is a constant that depends on $d$, the number of such pairs is at most $$T(n,m_2(d-1))=\frac {n^2}2\Big(1-\frac 1{m_2(d-1)}\Big)+O(1)$$
if $d\notin\{ 4,5\}$.
\end{thm}
\cite{EMP2}, Theorem 2 also considered a less restrictive variant of the notion of nearly $k$-distance sets. Let us denote by $W_k(d)$ the maximum cardinality $N$ such that for any $\varepsilon>0$ there exist $t_1\leq \dots \leq t_k$ and a set $S\subset \mathbb{R}^d$ with $|S|=N$ such that for any $p_1\neq p_2\in S$ we have
\[\|p_1-p_2\|\in \bigcup_{i=1}^k[t_i,(1+\varepsilon)t_i.]\]
In \cite{EMP2} they proved that $W_k(d)=(d+1)^k$, showing that ``weakly'' nearly $k$-distance sets can be much bigger than $k$-distance sets.

\subsection{Nearly $k$-distance sets}\label{sec13}
We say that a separated set of points $P$ is an \emph{$\varepsilon$-nearly $k$-distance set} with distances $1<t_1<\dots<t_k$ if
\[\|p_1-p_2\|\in\bigcup_{i=1}^k[t_i,t_i+\varepsilon]\] holds for any $p_1,p_2\in P$ with $p_1\neq p_2$. By analogy with $m_k(d)$, for
$k\geq 1$ and $d\geq  0$ let $M_k(d)$ denote the largest number $M$ such that for any $\varepsilon>0$ there exists an $\varepsilon$-nearly $k$-distance set in $\mathbb{R}^d$ of cardinality $M$.\footnote{Note that the case of $d=0$ is trivial: we have $m_k(0)=M_k(0)=1$ for any $k$. However, we need to introduce it  for technical reasons.} Obviously, $M_k(d)\ge m_k(d)$.  An expression equivalent to $M_k(d)$ occurs in \cite[page 19]{EMP2}, where they speculate that ``for k fixed, d sufficiently large probably $M_k(d)=m_k(d)$.'' We confirm this later.

Note that the difficulty in relating the maximal cardinalities of $k$-distance sets and nearly $k$-distance sets lies in the fact that, in nearly $k$-distance sets,  distances of different order of magnitude may appear. If we additionally assume that $\frac{t_{i+1}}{t_i}\leq K$ for some universal constant $K$ in the definition of nearly $k$-distance sets, a compactness argument would immediately imply that $m_k(d)$ equals this modified $M_k(d)$ (see Lemma~\ref{limit} below).

For $d,k\geq 1$ let $M_k(d,n)$ denote the largest number $M$ for which there is a  separated set $S\subseteq \mathbb{R}^d$ of $n$ points and $k$ real numbers $1\leq t_1\leq \dots \leq t_k$ such that the number of pairs $\{p_1,p_2\}$ with $p_1,p_2\in S$ satisfying
\begin{equation}\label{eq00}\|p_1-p_2\|\in\bigcup_{i=1}^k[t_i,t_i+1]\end{equation}
is at least $M$.

In these terms, Theorems~\ref{emps1}--\ref{emp3} determine (or asymptotically determine) the quantity $M_k(d,n)$ for $k=1,2$ and $d\geq 2$, as well as for $k\geq 1$ and $d=2$, provided that $n$ is large enough. It is natural to state the following general problem.
\begin{prb}\label{prb1}
  For any fixed $k,d\ge 1$ and $n\ge n_0(k,d)$, determine, at least asymptotically, the value of $M_k(d,n)$.
\end{prb}

\subsubsection{Flat sets}\label{sec131}

A \emph{subspace of a Euclidean space}
means a linear subspace. A \emph{plane of a Euclidean space} means an affine plane.

For reasons that appear to be technical, let us also introduce the following notions.
We usually use the notation $\Gamma$ (or $\Gamma_i$) to denote linear subspaces of $\mathbb{R}^d$, and $\Lambda$ (or $\Lambda_i)$ for affine planes.
The angle between a vector $v\neq 0$ and a non-zero linear subspace $\Gamma$ is the smallest angle that appears between $v$ and the vectors in $\Gamma\setminus \{0\}$. For two points $p,q\in \mathbb{R}^d$ we denote by $p-q$ the vector pointing from $q$ to $p$.

For $1\leq d\leq d'$ we say that a set of vectors $V\subseteq \mathbb{R}^{d'}\setminus \{0\}$ is \emph{$(d,\alpha)$-Flat} if there exists a linear subspace $\Gamma$ of dimension $d$ such that the angle between any $v\in V$ and $\Gamma$ is at most $\alpha$. (For $d=d'$
this is considered to be true for any
$\alpha\geq 0$.) If $\Gamma$ is such subspace, we say that $V$ is  \emph{$(d,\alpha)$-Flat with respect to $\Gamma$}.
Let $P\subseteq \mathbb{R}^{d'}$ be a set of points and $p$ be a point in $P$. We say that $P$ is \emph{$(p,d,\alpha)$-flat} (with respect to a linear $d$-subspace $\Gamma_p$) if  $\setbuilder{p-q}{q\in P\setminus \{p\}}$ is $(d,\alpha)$-Flat (with respect to $\Gamma_p)$.
We call a set $P$ \emph{$(d,\alpha)$-flat} if $P$ is $(p,d,\alpha)$-flat for every $p\in P$.
We say $P$ is \emph{globally $(d,\alpha)$-flat} if $\setbuilder{p-q}{p,q\in P, p\neq q}$ is $(d,\alpha)$-Flat.
If $|P|\leq 1$, then we define $P$ to be $(p,0,\alpha)$-flat (for any $p\in P$) and $(0,\alpha)$-flat.

Note that there is a difference between flatness and global flatness. 
For any $d\geq 2$ and $\beta< \arcsin d^{-1/2}$, $(d,\alpha)$-flatness for any $\alpha$
does not in general imply global $(d,\beta)$-flatness.
For an example for $d=2$, consider the following set in $\R^3$: $\{(0,0,1),(0,0,0),(K,0,0),(K,1,0)\}$, where $K=K_0(\alpha,\beta)$ is sufficiently large. However, if for some universal constant $K$ a set $S$ is $(p,d,\alpha)$-flat for some $p\in S$ and $\frac{\|p_1-p_2\|}{\|q_1-q_2\|}\leq K$ for each $p_1,p_2,q_1,q_2\in S$ with $q_1\neq q_2$, then $S$ is globally $(d,K'\alpha)$-flat, where $K'$ is a constant depending on $K$ and $d$.

For $0\leq d \leq d'$ let $N_k(d',d)$ be the largest number $N$ such that for every $\varepsilon,\alpha>0$ there exists a $(d,\alpha)$-flat $\varepsilon$-nearly $k$-distance set in $\mathbb{R}^{d'}$ of cardinality $N$. Note that $N_k(d',0)=1$. For $d\geq 1$ we denote $N_k(d):= N_k(d,d-1)$. Then we have $M_k(d-1)\le N_k(d)\leq  M_k(d)$. Indeed, any $\varepsilon$-nearly $k$-distance set in $\mathbb R^{d-1}$ is a $(d-1,0$)-flat $\varepsilon$-nearly $k$-distance set in $\mathbb R^d$.

Surprisingly, the behaviour of $M_k(d,n)$ is asymptotically determined by the value of $N_k(d)$ (see Proposition~\ref{propextend} and Theorem~\ref{manynearly}), thus the asymptotic solution of Problem~\ref{prb1} reduces to the following problem.
\begin{prb} For any $k,d\geq 1$ determine $N_k(d)$.
\end{prb}

Below, we state Conjecture~\ref{conj1}, which relates the behaviour of $M_k(d)$, $N_k(d)$  and $m_k(d)$.

\subsection{Constructions of nearly $k$-distance sets}\label{sec14}

In this subsection, we relate the quantities $M_k(d),  M_k(n,d), N_k(d)$, and $m_k(d)$.
For $d\geq 0$ and $k\geq 1$ let us define 
\begin{equation}\label{eqm'} M'_k(d):=\max \setbuilder{\prod_{i=1}^s m_{k_i}(d_i)}{\sum_{i=1}^sk_i=k, \sum_{i=1}^sd_i=d}.\end{equation}

\begin{gypo}\label{conj1} $N_k(d+1)=M_k(d)=M'_k(d)$ holds for all but finitely many pairs $k,d\ge 1$.
\end{gypo}
We do not have have any examples with $N_k(d+1)>M_k(d)$ or $M_k(d)>M'_k(d)$, and we believe that the first equality should always hold. However, there are constructions, that we will describe later, that suggest there could be some examples with $M_k(d)>M'_k(d)$. In Theorem~\ref{nearly k} we show that the conjecture holds for every $k$ and sufficiently large $d$.

.

\begin{prop}\label{propcombine} $M_k(d)\geq M'_k(d)$ holds for every $ k,d\ge 1$.
\end{prop}

\begin{proof}Let $\sum_{i=1}^sk_i=k$ and $\sum_{i=1}^sd_i=d$. Then there is an $\varepsilon$-nearly $k$-distance set in $\mathbb{R}^d$ of cardinality $\prod_{i=1}^s m_{k_i}(d_i)$ given by the following construction. For each $i$ let $S_i$ be a $k_i$-distance set in $\mathbb{R}^{d_i}$ of cardinality $m_{k_i}(d_i)$ and such that the distances in $S_i$ are much larger (in terms of $\varepsilon$) than the distances in $S_{i-1}$. Then $S_1\times\dots\times S_s$ is an $\varepsilon$-nearly $k$-distance set in $\mathbb{R}^d$ of cardinality $\prod_{i=1}^s m_{k_i}(d_i)$.
\end{proof}

{\bf Remark.} 
The example given above can be turned into a globally flat set in $\R^{d+1}$, providing a lower bound on $N_k(d+1)$. Actually, we do not know of any case when a flat set would provide a strictly better bound on $N_k(d)$ than a globally flat set. However, for some $k$ and $d$ there are extremal examples that are not  globally flat.  More precisely,
for any $\alpha,\varepsilon>0$ there is a $\beta>0$ depending only on $d$, such that we can construct a $(d-1,\alpha)$-flat $\varepsilon$-nearly $k$-distance set $S$ of cardinality $N_k(d)$ that is not globally $(d-1,\beta)$-flat.

Indeed, for example, for any $\alpha, \varepsilon > 0$ we can construct $(3,\alpha)$-flat $\varepsilon$-nearly $2$-distance sets of cardinality $N_2(4)=m_2(3)=6$ (see Theorem~\ref{nearly k}) in $\mathbb{R}^4$ as follows. Consider an equilateral triangle $\{p_1,p_2,p_3\}$ in $\mathbb{R}^4$ of side length $K$ spanning a $2$-dimensional plane $H$. For each $i\in [3]$ let $p_i-q_i$ be a vector of length $1$ orthogonal to $H$. It is not hard to check that $P=\{p_1,p_2,p_3,q_1,q_2,q_3\}$ is a $(3,\alpha)$-flat $\varepsilon$-nearly $2$-distance set if $K=K(\alpha,\varepsilon)$ is sufficiently large. However, if $p_1-q_1$ and $p_2-q_2$ are orthogonal, then $P$ is not globally $(3,\beta)$-flat, where $\beta$ can be taken to be $\pi/6$.

\vskip+0.1cm

We can give a lower bound on $M_k(d,n)$ in terms of $N_k(d)$. 
\begin{prop}\label{propextend}
For any fixed $k\geq 1, d\geq 2$ we have\begin{equation}
  M_k(d,n)\ge T\big(n,N_k(d)\big)\geq \frac {n^2}2\Big(1-\frac{1}{N_k(d)}\Big)+O(1).
\end{equation}
\end{prop}
The statement of Proposition~\ref{propextend} also holds with $N_k(d)$ replaced by $M_k(d-1)$, since $M_k(d-1)\leq N_k(d)$. 
\begin{proof} Let $\alpha,\varepsilon>0$ be sufficiently small, and $t_1>10n^2$. Consider a $(d-1,\alpha)$-flat $\varepsilon$-nearly $k$-distance set $S'\subseteq \mathbb{R}^d$ with distances $t_1\leq \dots \leq t_k$ of cardinality $N_k(d)$. For simplicity assume that $N_k(d)|n$. 

For each $p\in S'$, let $\Gamma_p$ be a $(d-1)$-dimensional subspace such that $S'$ is $(p,d-1,\alpha)$-flat with respect to $\Gamma_p$. Further, let $v_p$ be a unit vector that is orthogonal to $\Gamma_p$.
Replace each point $p\in S'$ with an arithmetic progression $A_p=\big\{p+i v_p:i\in \{1,\ldots, \frac{n}{N_k(d)}\}\big\}$.

If $|\sin\alpha|<\frac{1}{10n}$, then 
the distances between any point from $A_p$ and any point  from $A_q$ for $p\ne q,$ are within $1/2$ from the distance between $p$ and $q$. The set $S = \bigcup_{p\in S'}A_p$ has cardinality  $n$. Define a graph $G$ on $S$ by putting an edge between two points if their distance is in $[t_1-1/2,t_1+1/2]\cup\ldots\cup[t_k-1/2,t_k+1/2]$. Then $G$ is an $N_k(d)$-partite graph with equal parts. By definition, the number of edges in such graph is $T(n, N_k(d))$. This argument can easily be modified to deal with the case when $N_k(d)\not | n$.
\end{proof}

We point out the following difference between the case of $k=1$ and $k\geq 2$ of the known constructions with $M_k(d,n)$ nearly equal distances. Let $S\subseteq \mathbb{R}^d$ be a set of $n$ points and $1\leq t_1\leq \dots\leq t_k$ be reals such that the number of pairs $\{p_1,p_2\}$ with $p_1,p_2\in S$ and with $\|p_1-p_2\|\in \bigcup_{i=1}^k[t_i,t_i+1]$ is $M_k(d,n)$. For $k=1$ the known constructions are all of the type that was described in Proposition~\ref{propextend} with $S'$ being globally $(d-1,\alpha_n)$-flat with $\alpha_n \to 0$, and thus the normal vectors $m_v$ being almost parallel. However, this is not the case for $k=2$. For $k=2$, $d=4$, as explained before, there are $(3,\alpha)$-flat $\varepsilon$-nearly $2$-distance sets of cardinality $N_2(4)$ in $\mathbb{R}^4$ for any $\alpha, \varepsilon>0$ that are not globally $(3,\pi/6)$-flat, and hence the corresponding normal vectors $m_v$ are not pairwise almost parallel.\\

\textbf{Example from \cite{EMP2}.} The authors of \cite{EMP2} suggested that a construction in the same spirit as the one in Proposition~\ref{propextend}  should give a close to optimal bound for $M_k(d,n)$. With the two propositions above in hand, their construction is easy to describe: take $k_1,\ldots, k_s$, $d_1\le \ldots\le d_s$, such that $\sum_{i=1}^sk_i = k$ and $\sum_{i=1}^s d_i=d-1$. Next, represent the hyperplane $x_d=0$ as $\mathbb R^{d_1}\times\ldots\times \mathbb R^{d_s}$. In each $\mathbb R^{d_i}$, take the following $k_i$-distance set: either the set that gives the lower bound \eqref{eqlb} or, if $d_i=1$, an arithmetic progression of length $k_i+1$. Then combine the sets in the same way as in the proof of Proposition~\ref{propcombine}, obtaining a nearly $k$-distance set in the hyperplane $x_d=0$. Then extend it in $\mathbb R^d$ as in Proposition~\ref{propextend}. Assume that either $\ell=0$ and $d_1> 1$, or $\ell \geq 1$ and $d_{\ell}=1<d_{\ell+1}$, and we have chosen arithmetic progressions in the first $\ell$ subspaces. The obtained set has $\frac {n^2}2(1-\frac 1 Q+o(1))$ distances that fall in the $k$ intervals, where 
\begin{equation}\label{eqt} Q := (k_1+1)\cdot \ldots \cdot (k_{\ell}+1)\cdot {\binom{d_{l+1}+1}{ k_{l+1}}}\cdot\ldots\cdot {\binom{d_s+1}{k_s}}.\end{equation}
One then needs to optimise the value of $Q$ over all choices of $d_i,k_i$, $\ell$ and $s$. 
It is possible that $Q$ gives the value of $N_k(d)$ and $M_k(d+1)$ in many cases. Evidently, in order to maximise $Q$, one should take $k_1,\ldots, k_{\ell}$ to be nearly equal. 

We add the following observation.\vskip+0.1cm
{\bf Observation.} \emph{For any fixed $k,d$, there is a choice of $s,\ell,k_i,d_i$ that maximises $Q$ and such that $\ell = s-1$, that is, there is only one term of the form ${\binom{d_i+1}{k_i}}$.}\vskip+0.1cm

 Let us prove this. First, we observe that for any $i\ge l+1$, we may suppose that $k_i\le d_i/2$. Otherwise, reducing the number of distances does not decrease $Q$. We need the following claim.
\begin{cla}\label{cla1}
For any integers $a_1,a_2,z_1,z_2$ such that $z_1,z_2\ge 3$ and $a_1\le z_1/2,\ a_2\le z_2/2$, except $z_1=z_2=4$ and $a_1=a_2=2$, we have
\begin{equation}\label{eqbinprod}
{\binom{z_1}{a_1}}{\binom{z_2}{a_2}}\le {\binom{z_1+z_2-1}{ a_1+a_2}}.\end{equation}
\end{cla}
The proof is a simple calculation and is deferred to the appendix.
Using this claim, we can replace any pair of binomial coefficients with $d_i,d_j\ge 2$ ($i,j>l$) in \eqref{eqt} with one binomial coefficient without decreasing $Q$, unless both binomial coefficients are $\binom{4}{2}$. Moreover, if $d_i=1$ for $i>l$, then $k_i=1$ and we may simply replace $\binom{1+1}{1}$ by $(1+1)$, making it a term of the first type. Finally, if we have two terms of the form $\binom{4}{2}$, then we may replace them with $(2+1)\cdot\binom{6}{2}$, which is larger, and also uses $6$ dimensions and $4$ distances.\\

{\bf Examples with fixed $k$ or $d$.} It is not true that $M_k(d)=m_k(d)$ holds for every $k$ and $d$. There are several examples of $k$ and $d$ for which we need more than one multiplicative term to maximise \eqref{eqm'}, and hence $M_k(d)\geq M'_k(d)>m_k(d)$. Some of these examples we list below. When needed, we rely on the information from Table~\ref{small}.

\begin{itemize}
\item In $\mathbb{R}^2$ the largest cardinality of a $6$-distance set is $13$, while the product of two arithmetic progressions of length $4$ ($d_1=d_2=1$, $k_1=k_2=3$ in \eqref{eqt}) gives an $\varepsilon$-nearly $6$-distance sets of cardinality $16$. Thus $M_6(2)\geq M'_6(2)\geq 16>m_2(6)$.

\item In $\mathbb{R}^3$, the largest $4$-distance set has $13$ points, while we can construct $\varepsilon$-nearly $4$-distance sets of cardinality $15=3\cdot 5$ as a product of arithmetic progression of length $3$ and a $2$-distance set on the plane of cardinality $5$. Thus $M_4(3)\geq M'_4(3)\geq 15>m_4(3)$.
\item In $\mathbb{R}^2$ the cardinality of a $k$-distance set is $O (k\log k)$ by \cite{GK}, while the product of two  arithmetic progressions of length $(\lfloor k/2 \rfloor+1)$ and of length $(\lceil k/2 \rceil+1 )$ gives 
an $\varepsilon$-nearly $k$-distance set of cardinality $(\lfloor k/2 \rfloor+1)(\lceil k/2 \rceil+1 )\geq k^2/4$.

\item In $\mathbb{R}^d$ for $d\geq 3$ the cardinality of a $k$-distance set is 
$O\left (k^{(d/2)(d+2)/(d+1)}\right )$ by the result of Solymosi and Vu \cite{SV}. On the other hand, the product of $d$ arithmetic progressions of size $\lfloor k/d \rfloor+1$
gives an $\varepsilon$-nearly $k$-distance set of cardinality $(\lfloor k/d \rfloor +1)^d\geq (k/d)^d$. 
\end{itemize}

The largest $5$-distance set in $\mathbb{R}^2$ is of cardinality $12$. We may construct $\varepsilon$-nearly $5$-distance sets using product-type constructions as described in the list above, also of cardinality $12$. In addition, we can construct an $\varepsilon$-nearly $5$-distance set of size $12$ that is not of this product construction, and neither does it have the structure of a $5$-distance set. Take a large equilateral triangle, and in each of its vertices put a rhombus of a much smaller size with angles $\pi/3$ and $2\pi/3$ such that the angle of the corresponding sides of the rhombus and the triangle is $\pi/2$ as shown on Figure~\ref{fig1}. This example makes us suspect that there could be some exceptions to Conjecture~\ref{conj1}. Though we also believe there are only finitely many examples with $M_k(d)$ points that are not products of $k_i$-distance sets.
\begin{figure}[h]
\centering
{\includegraphics[width=12cm]{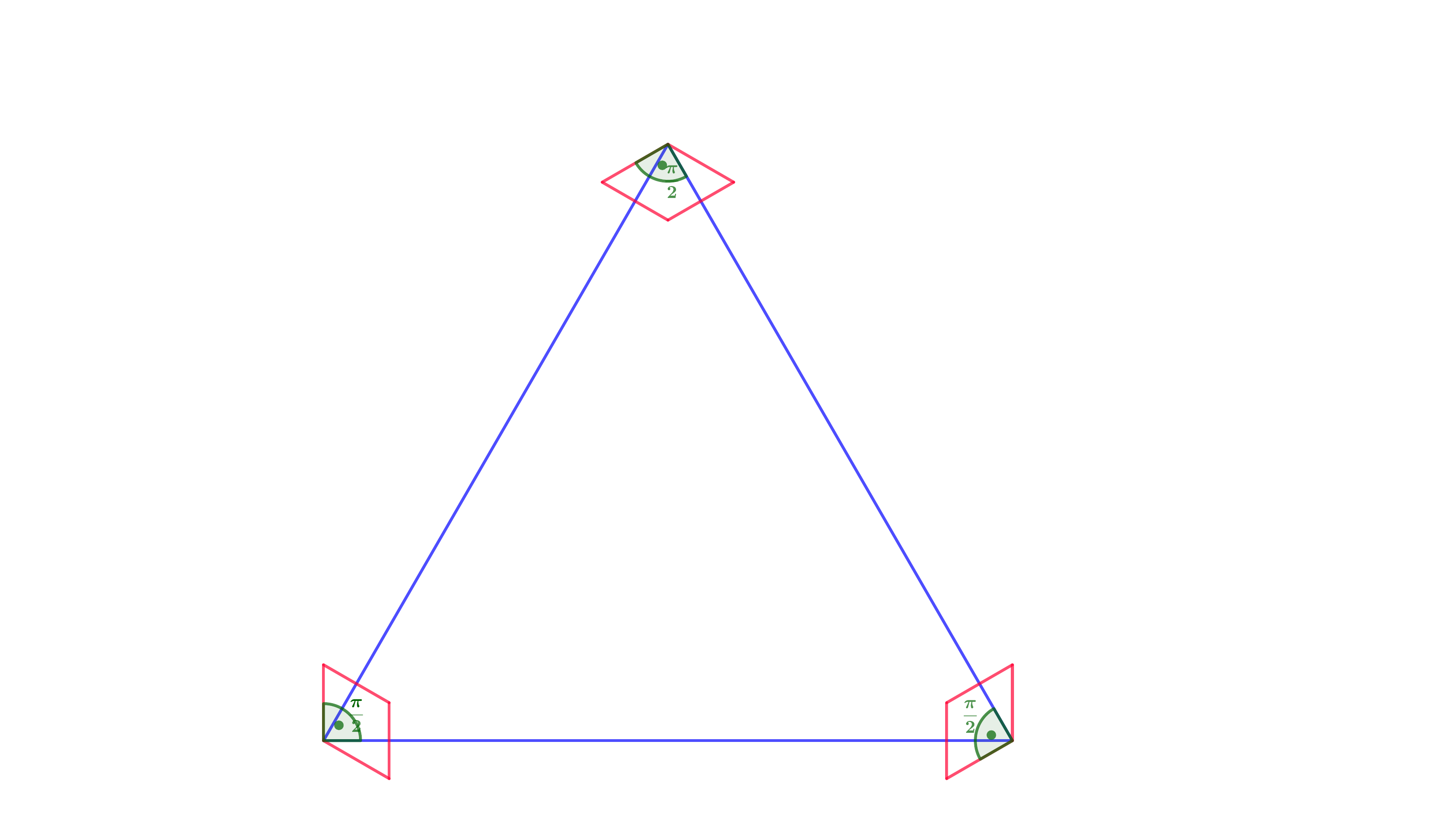}}
\caption{$\varepsilon$-nearly 5-distance set on the plane that is not product-type}\label{fig1}
\end{figure}

\subsection{Main results}\label{sec15}
Let us stress that all the sets that we consider in the paper are separated, which we assume tacitly for the rest of the paper. The first theorem deals with small values of $k$ and is one of the main results of the paper. 

\begin{thm}\label{nearly k} We have $N_k(d+1)=M_k(d)=m_k(d)$  for $d\geq 0$ if one of the following  holds:
\begin{itemize}
    \item[(i)] $d\ge d(k)$, where $d(k)$ is some constant depending on $k$;
    \item[(ii)] $k\le 3$.
\end{itemize} 
\end{thm}

For fixed $d$ and large $k$ we prove the following simple estimate. 
\begin{thm}\label{fixd} We have $M_k(d)=\Theta\left (k^{d}\right )$ and $N_k(d)=\Theta\left (k^{d-1}\right )$ for any fixed $d\geq 2$.
\end{thm}

We conjecture that $M_k(d)= (k/d)^d+o(k^d)$.

Another main result of the paper is the following theorem, which gives the promised relation between $N_k(d)$ and $M_k(d,n)$.

\begin{thm}\label{manynearly} For any $d\ge 2$, $k\ge 1$, $\gamma>0$  there exists $n_0$, such that for any $n\ge n_0$ we have
\begin{equation}\label{eq01} T(n,N_k(d))\leq M_k(d,n)\leq T(n,N_k(d))+\gamma n^2 \end{equation}
Moreover, \eqref{eq01} remains valid if in the definition of $M_k(d,n)$ we change the intervals of the form $[t_i,t_i+1]$ to intervals of the form $[t_i,t_i+cn^{1/d}]$ for some constant $c=c(k,d,\gamma)>0$.
\end{thm}

Theorem \ref{manynearly} combined with Theorem~\ref{nearly k} 
gives the value of $M_2(d,n)$, $M_3(d,n)$ and $M_k(d,n)$ for $d\ge d_0(k)$ asymptotically in terms of $m_2(d)$, $m_3(d)$ and $m_k(d)$.
In the mentioned cases, we can strengthen the result and 
determine the exact value of $M_k(d,n)$  for large $n$. 
In particular, this extends results of \cite{EMP2} (cf. Theorem~\ref{emp3}) to the cases of $d=4,5$.

\begin{thm}\label{thmsmallkexact} For $n\geq n_0(d,k)$ we have $$M_k(d,n) = T(n,m_k(d-1))$$ if either $k\le 3$ or $d\ge d(k)$. Moreover, the same holds with intervals of the form $[t_i,t_i+cn^{1/d}]$ for some $c=c(k,d)>0$.
\end{thm}
Again, in view of Proposition~\ref{propextend}, we only need to show that $M_k(d,n) \leq T(n,m_k(d-1))$.
This is a consequence of the more general Theorem~\ref{thmak}, presented in Section~\ref{sec6}.

Observe that $M_k(d,n)\leq T(n,M_k(d))$ is obvious from Tur\'an's theorem and the definition of $M_k(d)$. Hence the difficulty in proving Theorem~\ref{manynearly} lies in bounding $M_k(d,n)$ by the maximal cardinality of $(d-1,\alpha)$-flat nearly $k$-distance sets. Similarly, the difficulty in proving Theorem~\ref{thmsmallkexact} is bounding $M_k(d,n)$ by the maximal cardinality of $k$-distance sets in the space of one dimension smaller. 

We note that many different classes of dense geometric graphs were studied from a similar perspective. We mention diameter graphs \cite{Swan,Kup10,KP} and double-normal graphs \cite{MS,PS,Kup13}. In some cases, the relationship between the largest clique and the maximum number of parts in an arbitrarily large complete multipartite graph is quite intricate, as it is  the case for double-normal graphs, see \cite{Kup13}.

Note that an extended abstract \cite{FK} of this work with the same title was accepted for the Eurocomb 2019 conference.

\section{Proofs}\label{proofs}
The structure of this section is as follows. We start by giving sketches of the proofs of the theorems. The first subsection gives auxiliary results that are going to be used in the proof of Theorem~\ref{nearly k}, which is doubtless the hardest result in the paper. Most of this section can be skipped in the first reading. We summarise its content below. Subsections~\ref{sec22} and~\ref{sec23} are devoted to the proof of parts (i) and (ii) of Theorem~\ref{nearly k}, respectively. In subsection~\ref{sec24} we give a simple and short proof of Theorem~\ref{fixd}. This is largely independent from the previous material.
In Subsection~\ref{sec6} we give the proofs of Theorems~\ref{thmsmallkexact} and~\ref{manynearly}. These proofs are almost completely independent of the previous subsections, they basically only require Theorem~\ref{nearly k} as an input.\\

{\bf Subsection~\ref{sec21}: } The subsection starts with an important Lemma~\ref{clique} that allows us to find a very precise structure in the graph of distances between the points in case we can split distances into `small' and `large': it splits into clusters of points at small distances, and the distances between any two points from different clusters are large. 

Lemma~\ref{limit} is a compactness statement, which states that, in case the ratio of the largest and the smallest distance is uniformly bounded, then a nearly $k$-distance set converges to a usual $k$-distance set, and, in particular, is at most as big. This gives the intuition that non-trivial cases in the proof of Theorem~\ref{nearly k} deal with the case when $t_{i+1}/t_i>K$ for some $i$ and arbitrarily large $K$. The second part additionally tells us that, in case the ratio of two consecutive distances tends to $1$, then these distances are `glued', and in the limit we will have a $(k-1)$-distance set. 

Lemma~\ref{simple bound} gives a simple upper bound on $M_k(d)$ only based on Lemma~\ref{limit}.

Lemmas~\ref{flatslemma}-\ref{smallcliques} develop the machinery to deal with almost-flat sets. First, we translate the simple statement that if a vector (or a set) lies in several planes, then it lies in their intersection to the almost-flat setting. Somewhat surprisingly, this turns out to be quite tricky. Lemma~\ref{smallcliques} is a culmination of the subsection and is a key complement to Lemma~\ref{clique}. It translates to the almost-flat setting the following simple argument. Assume that we are given a nearly $k$-distance set, and there is an $i$  such that $t_{i+1}/t_i>K$ for some large $K$. Then, using Lemma~\ref{clique}, we can split it into `red clusters', such that inside the cluster each distance is at most $t_i,$ while the distance between any two points in different red cluster is at least $t_{i+1}.$ Moreover, assume that for any red cluster $R_i$ and a point $p$ outside it the distances between $p$ and points from $R_i$ are very close to the same $t_j,$ $j\ge i+1$ (for later reference in the sketches, let us call this `property $\star$'). Then it is easy to see that $R_i$ should be nearly orthogonal to $p-r$ for any $r\in R_i$. Thus, the intuition tells us that all red clusters should be close to planes that are orthogonal to the plane of $B$, where $B$ is a set of points that includes exactly one point from each red cluster. This is not difficult to check for a set $B$ that is `not flat', e.g., an almost-regular simplex, but it requires  preparations in general.

In the first reading, we recommend to read Lemma~\ref{clique}, as well as the statements of Lemmas~\ref{limit} and~\ref{smallcliques} and omit the rest.\\

{\bf Proof of Theorem~\ref{nearly k} (i): } The proof is by induction on $k.$ We look for the last `big jump' in the sequence of distances, and additionally induct on the position $i$ of this jump (the size of the jump being dependent on the position: the smaller $i$ is, the bigger the jump is). If $i=k$ (i.e., $t_k/t_{k-1}>K$) then we use Lemma~\ref{smallcliques}, even in the simple variant that we described above. In this case, the set $B$ forms an almost-regular simplex, and we conclude that the red clusters must be almost-flat w.r.t. the plane, orthogonal to the plane of $B$. We then use the fact that $m_{k-1}(d-j)m_1(j)<m_k(d)$ for large $d$. If $i<k,$ then the argument is mostly similar, however, it has an additional twist to deal with one complication. Namely, we may not be able to guarantee that property $\star$ holds if two large distances have very small difference and thus apply Lemma~\ref{smallcliques}. But then Lemma~\ref{limit} (ii) helps us out, telling that in the limit these two distances glue. If this is the case, we ignore the orthogonality structure that can be given by Lemma~\ref{smallcliques}, and apply a bound of the form $m_{i-1}(d)m_{k-i}(d)<m_k(d),$ valid for $d$ large w.r.t. $k.$

{\bf Proof of Theorem~\ref{nearly k} (ii): } The proof follows a very similar outline as part (i). The case of two distances uses the same ideas as part (i), even in a simpler form. The case of three distances, however, poses some complications. First, if $t_3/t_2>K$ then the proof is as in the case $i=k$ above. If $t_2/t_1>K>t_3/t_2>1+2/K$ then the proof is as above for $i<k$, but with property $\star$, i.e., our main tool is the `almost orthogonal decomposition' via Lemma~\ref{smallcliques}. If $t_3/t_2<1+2/K$ then the two largest distances glue, and the argument as above allows us to finish the proof for $d\geq 6$. 
However, this is not sufficient for small $d,$ which we deal with separately. We have two subcases in this situation, each of which allow us to get more information on our nearly $3$-distance set. The first subcase is when $t_1\gg t_3-t_2.$ In that case, we are again able to guarantee the almost-orthogonality of red clusters to the almost-regular simplex $B$ (cf. Figure~2) and again finish using Lemma~\ref{smallcliques}. The second subcase is when $t_1$ is not so large w.r.t. $t_3-t_2.$ We are then able to find additional structure in between any two red clusters, say $R_i$ and $R_{\ell}$. Namely, if both distances $t_2,t_3$ appear from $r\in R_i$ to points in $R_{\ell}$, then at most one distance can appear between a point $r'\in R_{\ell}$ and any point in $R_i.$ This turns out to be sufficient to settle this case. Such `bizarre' cases as the last one also should give a hint on why it is difficult to extend the result to more distances.

We also note that an additional complication in both proofs is that we had to work with a notion of almost-flat sets (defined at the end of Section~\ref{sec21}) instead of flat sets. This is only needed in order to prove Theorem~\ref{thmsmallkexact}.\\

We omit the sketch of proof of Theorem~\ref{fixd}, which is short and simple, and go on to the Turán-type results.\\

{\bf Proof of Theorem~\ref{thmsmallkexact}: } The proof essentially uses certain supersaturation-type results for the Turán's theorem, and then refines the geometric structure of the relations between the points until we get a contradiction. We argue indirectly. 
Let $\ell-1$ be the largest size of an almost-flat $\varepsilon$-nearly $k$-distance set.  The main (and the one that is less standard) extremal graph theory tool is the following supersaturation result of Erd\H os that states that, once the number of edges in an $n$-vertex graph is at least $T(n,\ell-1)+1$, it contains a positive fraction of all $(\ell-2)$-cliques containing some edge $e.$ Using this and some other results and ideas from exremal graph theory, we start with a graph that has one more edge than the theorem states and find the following configuration: a multipartite graph $K_{1,1,2,\ldots,2}$ with $\ell$ parts, where any two points in the configuration lie at a distance $\gg cn^{1/d}$, the distances between any two fixed parts all fall into the same interval $[t_i,t_i+cn^{1/d}],$ and any angle $p_ip_jq_i$ for $p_i,q_i$ belonging to the $i$-th part and $q_j$ belonging to the $j$-th part, $i\ne j$, is at most $\alpha$. Then, after rescaling, we get a separated set with all distances falling into $[t_i',t_i'+\varepsilon]$ for some $i$ (thus the need for point of the  set before the scaling to be far apart) and that can be shown to be almost-flat w.r.t. some hyperplanes. 

Let us comment on the almost-flatness. Almost -flatness requires that the local flatness condition is satisfied for all but at most two vertices of the graph, and the reason for almost-flatness here (and, as a consequence, additional complications in the proof of Theorem~\ref{nearly k}) is the result of Erd\H os that we cited. It does not guarantee a positive density of $\ell$-cliques once we add an extra vertex (to do so, one has to add $\gamma n^2$ extra edges, as in Theorem~\ref{manynearly}), it only guarantees a positive density of $(\ell-2)$-cliques sharing some edge. As a result, the final graph has the first two parts of size $1$, and we cannot ensure the local flatness condition in these vertices.

The proof of Theorem~\ref{manynearly} follows the same logic and thus we omit it here.

\subsection{Auxiliary lemmas}\label{sec21}
\begin{lem}\label{clique} Let $S\subseteq \mathbb{R}^d$ be a finite set. Assume that for every $p_1,p_2\in S$ with $p_1\neq p_2$ the pair $\{p_1,p_2\}$ is coloured with red or blue, such that the distance between the points in any blue pair is strictly more than $3$ times as big as the distance between any red pair. If $B$ is a largest blue clique in $S$, then $S$ can be partitioned into $|B|$ vertex-disjoint red cliques $R_1,\dots,R_{|B|}$ having the following properties.
\begin{enumerate}
\item Each $R_i$ shares exactly one vertex with $B$.
\item If $p\in R_i$, $q\in R_j$ and $i\ne j$, then $\{p,q\}$ is blue.
\end{enumerate}
\end{lem}

\begin{proof} Take a largest blue clique $B=\{p_1,\ldots,p_s\}$. Construct $R_i$ by including in it $p_i$ and all the points that form a red pair with $p_i$. By the triangle inequality each $R_i$ is a red clique. Further, by the maximality of $B$, each point from $S$ forms a red distance with at least one point in $B$, and thus $R_1,\ldots, R_b$ cover $S$. Next, they are disjoint: if $p\in R_i\cap R_j$, then both $\{p,p_i\}$ and $\{p,p_j\}$ are red, which by triangle inequality implies that either $i=j$ or that $\{p_i,p_j\}$ is red (but the second possibility contradicts the definition of $B$). 
Finally, if $p\in R_i,$ $q\in R_j$, $i\ne j$, then $\{p,q\}$ must be blue by the triangle inequality: otherwise $\|p_i-q_j\|\le \|p_i-p\|+\|p-q\|+\|q-p_j\|$, and if all the pairs on the right are red, then $\{p_i,p_j\}$ is red.
\end{proof}

Note that a statement similar to that of Lemma \ref{clique} was also used in \cite{EMP}. The next lemma follows by a standard compactness argument.

\begin{lem}\label{limit}Let $S_1,S_2,\dots$ be a sequence such that $0\in S_i$ is an $\varepsilon_i$-nearly $k$-distance set in $\mathbb{R}^{d'}$ with distances $1\leq t_{i,1}<\dots<t_{i,k}$ and with $\varepsilon_i\to 0$. Further, let $\alpha_i$ be a sequence with $\alpha_i\to 0$. Then the following is true.
\begin{enumerate}
\item[(i)] If $k=1$ or $k\geq 2$ and there is a $K$ such that $\sup_i\max_{1\leq j<k}\frac{t_{i,j+1}}{t_{i,j}}\leq K$, then we have \mbox{$\limsup_{i\to \infty} |S_i|\leq m_k(d')$.} If additionally there is a $0\leq d\leq d'$ such that for every $i$ the set $S_i$ is $(p_i,d,\alpha_i)$-flat for some $p_i\in S_i$, then $\limsup_{i\to \infty} |S_i|\leq m_k(d)$.

\item[(ii)] If $k\geq 2$ and there is a $K$ such that $\sup_i\max_{1\leq j<k}\frac{t_{i,j+1}}{t_{i,j}}\leq K$ and 
for some $1\le r\le k-1$ we have $\lim_{i\to\infty}\frac{t_{i,r+1}}{t_{i,r}}=1$, then $\limsup_{i\to \infty} |S_i|\leq m_{k-1}(d')$. If additionally for every $i$  $S_i$ is $(p_i,d,\alpha_i)$-flat for some $p_i\in S_i$, then $\limsup_{i\to \infty} |S_i|\leq m_{k-1}(d)$.

\end{enumerate}
\end{lem}

\begin{proof}
  We only give details of the proof of (ii), part (i) can be done similarly. We start with the first part of the statement. Take any sequence $S_1,S_2,\ldots, $ satisfying the conditions and scale each $S_i$ by $\frac{1}{t_{i,1}}$. Abusing notation, we denote the new sets $S_i$ as well. Then the condition  $\sup_i\max_{1\leq j<k}\frac{t_{i,j+1}}{t_{i,j}}\leq K$ implies that there is an absolute $R>0$ such that each $S_i$ is contained in a ball $B$ with centre $0$ and of radius $R$. A volume argument implies that there exists an $M_K$ such that $|S_i|\leq M_K$ for all $i$. Take an infinite subsequence of $S_1,S_2\ldots$ in which all sets have fixed cardinality $M\leq M_K$. Using the compactness of $\underbrace{B\times\ldots\times B}_{M\text{ times}}$, select out of it a subsequence $S_{i_1},S_{i_2},\ldots$  that pointwise converges to the set $S:=\{P_1,\ldots, P_M\}\subset B$ with distances $T_1,\ldots,T_k$, and where $T_j = \lim_{s\to\infty} \frac{t_{i_s,j}}{t_{i_s,1}}$. Note that $T_{r+1}=T_r$ due to the assumption $\lim_{i\to\infty}\frac{t_{i,r+1}}{t_{i,r}}=1$. Thus $S$ is a $(k-1)$-distance set, and so $M = |S|\le m_{k-1}(d').$
  
 Let us next show the second part of the statement. Taking the set $S$ as above, we obtain that it must additionally be $(d,0)$-flat. This means that $S$ lies in a $d$-plane, thus $M= |S|\le m_{k-1}(d).$
 \end{proof}

The statement below allows us to get a grip on $M_k(d)$.

\begin{lem}\label{simple bound} For any $1\leq k$, and $0\leq d\leq d'$ we have \[N_k(d',d)\leq f(d,k)=\max \setbuilder{\prod_{i=1}^s m_{k_i}(d)}{\sum_{i=1}^s k_i=k}.\] In particular $M_k(d)<\infty$.
\end{lem}
Note the difference in the definition of $M'_k(d)$ and the function $f$ above.

\begin{proof}[Proof]
First note that $f$ satisfies $f(d,k_1+k_2)\ge f(d,k_1)f(d,k_2)$ for any $1\leq k_1,k_2$. 

Let $S$ be a $(d,\alpha)$-flat $\varepsilon$-nearly $k$-distance set in $\mathbb{R}^d$ with distances $1\leq t_1<\dots<t_k$ and with sufficiently small $\alpha,\varepsilon$. We need to show that $|S|\le f(d,k).$ For each $d$ we induct on $k$.

If $\frac{t_i}{t_{i-1}+\varepsilon}\leq 3$ holds for every $1<i\leq k$ (or if $k=1$), then by Lemma~\ref{limit} (i) we have $|S|\le m_k(d)\le f(d,k)$. Otherwise, let $i$ be the largest index such that $\frac{t_i}{t_{i-1}+\varepsilon}> 3$. For every $p_1,p_2\in S$ with $p_1\neq p_2$ colour the pair $\{p_1,p_2\}$ with blue if $\|p_1-p_2\|\geq t_i$ and with red otherwise. Let $B$ be a largest blue clique in this colouring. By induction, $|B|\leq f(d,k-i+1)$ if $\alpha$ and $\varepsilon$ are sufficiently small. Next, by Lemma~\ref{clique}, $S$ can be covered by $|B|$ vertex disjoint red cliques $R_1,\dots,R_{|B|}$. By induction again, the cardinality of any red clique is at most $f(d,i-1)$, thus \begin{equation*}|S|\leq f(d,k-i+1)f(d,i-1)\le f(d,k).\end{equation*}
\end{proof}

The next three statements 
describe some cases when $\alpha$-flatness with respect to different subspaces can be ``combined'' into $\alpha$-flatness with respect to a smaller-dimensional subspace. 
For a linear subspace $\Gamma$ we denote by $\Gamma^{\bot}$ the orthogonal complement of $\Gamma$, and for a vector $v\in \mathbb{R}^d\setminus \{0\}$ we denote by $v^{\bot}$ the $(d-1)$-dimensional subspace orthogonal to $v$.

\begin{lem}\label{flatslemma} For any $\gamma'>0$ there exists $\beta_0>0$ such that the following is true for all $0<\beta\leq \beta_0$ and all $d\geq 1$. Let $R\subseteq \mathbb{R}^d$ be a set of points, and let $p\in R$. Further, let $V\subset \mathbb{R}^d\setminus \{0\}$ be a set of vectors such that for every $v\in V$ the set $R$ is \mbox{$(p,d-1,\beta^d)$-flat} with respect to $v^{\bot}$. 
If $j\geq 1$ is the smallest integer for which $V$ is $(j,\beta^j)$-Flat, 
then $R$ is $(p,d-j,\gamma')$-flat.
\end{lem}

\begin{proof}We are going to prove that for every $\gamma'>0$ there exists  $\beta_0 >0$ such that the following is true for every $0<\beta\leq \beta_0$. Let $w\in \mathbb{R}^d\setminus \{0\}$ be a vector, and let $V\subset \mathbb{R}^d\setminus \{0\}$ be a set of vectors, such that for every $v\in V$ the angle between $w$ and $v^{\bot}$ is at most $\beta^d$. (This is equivalent to $|\langle w/\|w\|,
v/\| v\| \rangle |\leq \beta^d$, hence also to that the angle between $v$ and $w^{\bot}$ is at most $\beta^{d}$.) If $j\geq 1$ is the the smallest integer for which $V$ is $(j,\beta^j)$-Flat, and $V$ is $(j,\beta^j)$-Flat with respect to $\Gamma$, then the angle between $w$ and $\Gamma^{\bot}$ is at most $\gamma'$. Applying this for every vector of the form $w=q-p$, where $q\in R\setminus \{p\}$, implies the statement, since $\Gamma'$ is of dimension $d-j$. 

Arguing indirectly, assume that for every $\beta_0$ there is a $\beta<\beta_0$ such that the angle between $w$ and $\Gamma^{\bot}$ is larger than $\gamma'$. We will show that then $V$ is $(j-1,\beta^{j-1})$-Flat with respect to $\Gamma\cap w^{\bot}$. If $\beta$ is sufficiently small, then $w\notin \Gamma^{\bot}$, thus the dimension of $\Gamma\cap w^{\bot}$ is $j-1$, contradicting the minimality assumption on $j$.

We may assume that $\|w\|=1$ and $\|v\|=1$ for every $v\in V$.
Let $\{w_1,\dots,w_{d}\}$ be an orthonormal basis of $\R^d$, where additionally $\{w_1,\ldots,w_{j-1}\}$ is an (orthonormal) basis of $\Gamma\cap w^{\bot}$, further $\{w_1,\dots, w_j\}$ is a basis of $\Gamma$, and  $\{w_{j+1},\dots,w_d\}$ is a basis of $\Gamma^{\bot}$. Then $w$ can be written as 
$w=\eta_jw_j+\dots+\eta_dw_d$, where 
\begin{equation}\label{sinus}
|\eta_j|> \sin \gamma',
\end{equation}
since $w$ has angle larger than $\gamma'$ with $\Gamma^{\bot}$. Next, any $v\in V$ can be written as 
$\theta_1 w_1+\dots+\theta_d w_d$, where \[\theta_1^2+\dots+\theta_j^2\geq \cos^2 (\beta^j),\] 
since $v$ has an angle at most $\beta^j$ with $\Gamma$. Further, we have 
\[|\langle v, w\rangle|=|\theta_j \eta_j+\theta_{j+1}\eta_{j+1}\dots +\theta_d \eta_d|\leq\beta^d,\] 
since the angle of $w$ and $v$ is in $[\frac{\pi}{2}-\beta^d,\frac{\pi}{2}+\beta^d]$. By the Cauchy-Schwarz inequality we  have 
$$|\theta_{j+1}\eta_{j+1}+\dots +\theta_d \eta_d|\leq \sqrt{\theta_{j+1}^2+\ldots+\theta_d^2}\cdot \|w\|\le \sqrt{1-\cos^2 (\beta^j)}=\sin(\beta^j)\le \beta^j.$$  By the triangle inequality and the previous two inequalities, we get  
\[|\theta_j \eta_j|\leq |\theta_j \eta_j+\dots +\theta_d \eta_d|+|\theta_{j+1}\eta_{j+1}\dots +\theta_d \eta_d|\le 2\beta^j.\] If $\beta<\frac{\sin^2\gamma'}{4}$, then the inequality above together with \eqref{sinus} implies $|\theta_j|=|\theta_j\eta_j|/|\eta_j|\leq 2\beta^j/\sin \gamma'<\beta^{j-0.5}$. Thus, if $\beta$ is sufficiently small, then
\[\theta_1^2+\dots+\theta_{j-1}^2\geq \cos^2 (\beta^j)-\beta^{2j-1}\ge \cos^2 (\beta^{j-1}),\]
where the last inequality follows from the fact that $\cos\theta= 1-(\frac{1}{2}+o(1))\theta^2$ for small $\theta$. This means that the  angle between  $v$ and $\Gamma\cap w^{\bot}$ is at most  $\beta^{j-1}.$  Since this is valid for any $v\in V$, we conclude that $V$ is $(j-1,\beta^{j-1})$-flat with respect to $\Gamma\cap w^{\bot}$, a contradiction.
\end{proof}

\begin{lem}\label{smallcliques}
For any $d'$ and $\gamma>0$ there exist $\beta'>0$ such that the following is true for any $0<\beta\leq \beta'$, sufficiently small $0<\alpha\leq \alpha(\beta)$, and sufficiently large $K\geq K(\alpha,\beta)$. Let $B\cup R\subseteq \mathbb{R}^{d'}$ be a separated set with  $B\cap R=\{b\}$, and with the following two properties.
\begin{enumerate}
    \item For any $r,r'\in R$ with $r\neq r'$ and for any $b'\in B\setminus \{b\}$ we have $K\|r-r'\|\leq \|b-b'\|$. 
    \item For any $b'\in B\setminus \{b\}$ there is a number $t>0$ such that for any $r\in R$ we have $\|b'-r\|\in [t,t+\beta^{d'+1}]$.
\end{enumerate}
Further, let $j\geq 1$ be the lowest dimension such that $B$ is $(b,j,\beta^j)$-flat. Assume that for some $r\in R$ and $d\leq d'$ the set $B\cup R$ is $(r,d,\alpha)$-flat.  Then $R$ is $(r,d-j,\gamma)$-flat.
\end{lem}

\begin{proof}
For $|R|\leq 1$ the statement holds by definition. So we will
suppose that $|R| \geq 2$. Let $\beta_0$ as in Lemma \ref{flatslemma} with $\gamma'=\frac{\gamma}{2}$, and let $\beta'\leq \beta_0$ be sufficiently small, to be further specified later. Let $r\in R$ be any point of $R$. Assume that $B\cup R$ is $(r,d,\alpha)$-flat with respect to a $d$-dimensional subspace $\Gamma_r$. Let $\pi_r(V)$ be the projection of $V:=\setbuilder{b'-r}{b'\in B\setminus \{b\}}$ on $\Gamma_r$. Assume that $j_r>1$ is the lowest dimension such that $\pi_r(V)$ is $(j_r,\frac{\beta^{j_r}}{2})$-Flat. Further, let $\Gamma_r'$ be a subspace of $\Gamma_r$ of dimension $j_r$ such that $\pi_r(V)$ is $(j_r,\frac{\beta^{j_r}}{2})$-Flat with respect to $\Gamma_r'$.

By Condition 1, for any $b'\in B\setminus \{b\}$ the angle between $(b'-b)$ and $(b'-r)$ is at most $\alpha$ if $K$ is sufficiently large. The angle between $(b'-r)$ and $\pi_r(b'-r)$
is at most $\alpha$. Further, the angle between $\pi_r(b'-r)$ and $\Gamma'_r$ is at most $\frac{\beta^{j_r}}{2}$. These, together with the triangle inequality imply that $B$ is $(b,j_r,\beta^{j_r})$-flat with respect to $\Gamma'_r$ if  $2\alpha+\frac{\beta^{j_r}}{2}\leq \beta^{j_r}$. By the minimality of $j$, it follows then that $j_r\geq j$.

Let $\Lambda_r$ be the affine plane in $\mathbb{R}^{d'}$ of dimension $j_r$ through $r$ parallel to $\Gamma_r$. For a point $p\in \mathbb{R}^{d'}$ we denote by $\pi'_r(p)$ the projection of $p$ on $\Lambda_r$. Similarly, for a set $X\subseteq \mathbb{R}^{d'}$, let $\pi'_r(X)$ denote the projection of $X$ on $\Lambda_r$. Note that $\pi'_r(r)=r$. Let $(\pi_r'(b')-r)^{\bot_r}$ denote the $(d-1)$-dimensional subspace in $\Gamma_r$ that is orthogonal to $\pi_r'(b')-r$. Note that the vector $\pi_r'(b')-r$ lies in $\Gamma_r$.
\begin{cla}
For every $b'\in B\setminus \{b\}$ and $r\in R$, the projection $\pi'_r(R)$ is $(r,d-1, \beta^{d})$-flat with respect to $(\pi_r'(b')-r)^{\bot_r}$ if $\alpha$ and $\beta$ are sufficiently small and $K$ is sufficiently large. 
\end{cla}
\begin{proof}
Let $r'\in R\setminus \{r\}$ be any point of $R$. Condition 2 gives that \mbox{$\big | \|b'-r\|-\|b'-r'\| \big | \leq 2\beta^{d'+1}$}. Since $1\leq \|r-r'\|\leq \frac{t+\beta^{d'+1}}{K}$ , we obtain that  $\angle b'rr'\in [\frac{\pi}{2}-\frac{\beta^{d'}}{2},\frac{\pi}{2}+\frac{\beta^{d'}}{2}]$, if $\beta$ is sufficiently small and $K$ is sufficiently large. Further, we have $\angle r'r\pi'_r(r')\leq \alpha$ and $\angle\pi_r'(b')rb'\leq \alpha$. Thus, $\angle \pi_r'(b')r\pi'_r(r')\in [\frac{\pi}{2}-\beta^{d},\frac{\pi}{2}+\beta^{d}]$ if $\alpha$ is sufficiently small.
Since $\pi'_r(R)$ is contained in $\Lambda_r$, we obtain that for any $b'\in B$ the set $\pi'_r(R)$ is $(r,d-1,\beta^{d})$-flat with respect to $(\pi_r'(b')-r)^{\bot_r}$.
\end{proof}

Now we apply Lemma \ref{flatslemma} with $\pi_r(V)$ and $\gamma'=\frac{\gamma}{2}$, and obtain that $\pi'_r(R)$ is $(r,d-j_r,\frac{\gamma}{2})$-flat. The inequality $j_r\geq j$ implies that $\pi'_r(R)$ is $(r,d-j,\frac{\gamma}{2})$-flat. Since for any $r'\in R$ the angle between $r'-r$ and $\pi_r'(r')-r$ is at most $\alpha$, it follows that $R$ is $(r,d-j,\gamma)$-flat if $\alpha$ is sufficiently small.
\end{proof}

The proof of the following lemma is a simple calculation.

\begin{lem}\label{global} Let $S\subseteq \mathbb{R}^d$ be a set such that $\frac{\|p_1-p_2\|}{\|q_1-q_2\|}\leq K$ holds for any $p_1,p_2,q_1,q_2\in S$ with $q_1\neq q_2$. If $S$ is $(p,j,\alpha)$-flat for some $p\in S$, then $S$ is $(q,j,20(K\alpha)^{1/2})$-flat for any $q\in S$.
\end{lem}

\begin{proof}
Let $S$ be $(p,j,\alpha)$-flat with respect to $\Gamma$. We will show that for any $q,r\in S$ with $q\neq r$ there is a vector $v\in \Gamma$ such that the angle between $q-r$ and $v$ is at most $20(K\alpha)^{1/2}$.

Let $v_q,v_r\in \Gamma$ be vectors such that the angle between $q-p$ and $v_q$, and the angle between $r-p$ and $v_r$ is at most $\alpha$. Further, assume that $\|q-p\|=\|v_q\|$ and $\|r-p\|=\|v_r\|$. The following claim, whose proof is deferred to the Appendix, finishes the proof.

\begin{cla}\label{beta} The angle between $(v_q-v_r)\in \Gamma$ and $q-r$ is at most $20(K\alpha)^{1/2}$.
\end{cla}
\end{proof}

We need the following seemingly technical variant of $\alpha$-flatness, which is however crucial for proving Theorem~\ref{thmsmallkexact}. For $d\geq 1$ we say that $P$ is \emph{almost $(d,\alpha)$-flat} if $P$ is $(p,d,\alpha)$-flat for all but at most two $p\in P$. Note that this means if $P$ is almost $(0,\alpha)$-flat and $\alpha\leq \frac{\pi}{2}$, then $|P|\leq 2$. 
We also use the convention that for any $\alpha>0$ a set $P$ is $(0,\alpha)$-flat if $|P|\leq 2$.
Let $A_k(d',d)$ denote the largest number $A$ such that for any $\varepsilon,\alpha>0$ there exists an almost $(d,\alpha$)-flat $\varepsilon$-nearly $k$-distance set in $\mathbb{R}^{d'}$ of cardinality $A$. Note that $A_k(d',0) = 2.$ For $d\geq 1$ let $A_k(d)= A_k(d,d-1)$.

Let us 
summarise the trivial inequalities between the different parameters we introduced: 
\begin{equation}\label{eqrelation} m_k(d)\le M_k'(d)\le M_k(d)\le N_k(d',d)\le A_k(d',d)\leq  M_k(d'),\end{equation}
for any $d'\geq d\geq 0$.

\subsection{Proof of Theorem~\ref{nearly k} (i)}\label{sec22}
We will prove that for any $k\geq 1$ if $d$ is sufficiently large compared to $k$, then for any $d'\geq d$ we have $A_k(d',d) = m_k(d)$. This is sufficient in view of \eqref{eqrelation}.
We induct on $k$. The case $k=1$ is implied by Lemma~\ref{limit} (i). Assume that the statement of Theorem~\ref{nearly k} is true for every $k'\leq k-1$ with $d>D_{k'}$. We will prove the statement for $k$ and $d>D_k$, where the quantity $D_k$ is chosen later.

For $K>0$ and for an $\varepsilon$-nearly $k$-distance set $S$ with distances $1\leq t_1< \dots <t_k$ let $\phi_S(K)=1$ if $\max_{1< i\leq k}\frac{t_{i}}{t_{i-1}+\varepsilon}< K$, and otherwise let $\phi_S(K)$ be the largest index $1< i\leq k$ such that $\frac{t_{i}}{t_{i-1}+\varepsilon}\ge K$.

\begin{lem}\label{induction} 
If $\varepsilon_k$ and $\alpha_k$ are sufficiently small and $d>D_k$ for some sufficiently large $D_k$, then the following is true for every $1\leq i \leq k$. There exist $K_i\geq K_{i+1}\geq \ldots\geq K_k$ such that if $S$ is an almost $(d,\alpha_k)$-flat $\varepsilon_k$-nearly $k$-distance set in $\mathbb{R}^{d'}$, and for some $j\geq i$ we have $\phi_S(K_j)\geq j$, then $|S|\leq m_k(d)$.
\end{lem}

Note that since $\phi_S(K)\geq 1$ holds for any $K$, Lemma~\ref{induction} with $i=1$ implies the theorem. We phrased Lemma~\ref{induction} in this seemingly strange form, because this way it is convenient to prove it by backwards induction on $i$.

\begin{proof} The proof is by backwards induction on $i$. We start by showing that the statement is true for $i=k$ with some sufficiently large $K_k\geq 4$.

Assume that $\phi_S(K_k)=k$ and for every $p_1,p_2\in S$ with $p_1\neq p_2$ colour the pair $\{p_1,p_2\}$ with blue if $\|p_1-p_2\|\in [t_k,t_k+\varepsilon_k]$ and with red otherwise. Let $B$ be a largest blue clique in $S$. Then $S$ can be partitioned into $|B|$ red cliques $R_1,\dots,R_{|B|}$ as in Lemma~\ref{clique}. Let $R$ be a largest red clique and let $R\cap B=\{b\}$. We will apply Lemma~\ref{smallcliques} to $R\cup B$ with a sufficiently small
$\gamma$ (to be chosen later) to bound $|R|$.

Let $\beta'$ be as in Lemma~\ref{smallcliques}, and $0<\beta\leq \beta'$ be sufficiently small, to be set later. We may assume that $\alpha_k\leq \alpha(\beta)$ and choose $K_k$ such that $K_k\geq 2K(\alpha_k,\beta)$. We may apply Lemma~\ref{smallcliques}, as Conditions $1,2$ are automatically satisfied if $\varepsilon_k\le \beta^{d'+1}$. Let $j$ be as in the Lemma with $\alpha=\alpha_k$. Since $S$ is almost $(d,\alpha_k)$-flat, we have that $R\cup B$ is $(r,d,\alpha_k)$-flat for all but at most two $r\in R$. If $r\in R$ is such that $R\cup B$ is $(r,d,\alpha_k)$-flat, then by Lemma~\ref{smallcliques} we obtain that $R$ is $(r,d-j,\gamma)$-flat.

Thus, if $\gamma\leq \alpha_{k-1}$ then we have $|R|\leq A_{k-1}(d',d-j)$. Note also that $B$ is $(j,20\beta^{j/2})$-flat by Lemma \ref{global}, thus $|B|\leq m_1(j)=j+1$ by Lemma~\ref{limit} (i) if $\beta$ and $\varepsilon_k$ are sufficiently small.
These imply that
\[|S|\leq |B||R|\leq (j+1)A_{k-1}(d',d-j).\]
We separate two cases in order to bound $(j+1)A_{k-1}(d',d-j)$.

\medskip

\bf Case 1: \rm $d-j\geq D_{k-1}$. In this case we obtain 
\begin{equation*}
(j+1)A_{k-1}(d',d-j)
\leq  (j+1)\binom{d-j+k-1}{k-1}\leq \binom{d+1}{k}\leq m_k(d),
\end{equation*}
where the first inequality is true by induction and by \eqref{eqbbs}, the second is true if $d$ is sufficiently large, and the third is true by \eqref{eqlb}.

\medskip

\bf Case 2: \rm $d-j< D_{k-1}$. In this case  we have  $A_{k-1}(d',d-j)\leq 2+N_{k-1}(d',d-j)$. By Lemma~\ref{simple bound} we have  $N_{k-1}(d',d-j)\leq C_k$ where $C_k$ depends only on $k$ and $D_{k-1}$, hence depends only on $k$. Thus we obtain
\[ (j+1)A_{k-1}(d',d-j)\leq (j+1)(C_k+2)\leq \binom{d+1}{k}\leq m_k(d),\]
where the second inequality is true $d$ is sufficiently large, and the third inequality is \mbox{true by \eqref{eqlb}.}

\medskip

We now turn to the induction step.

Assume that the statement holds for every $i+1$ with $K_{i+1}\geq K_{i+2}\geq \dots \geq K_k$, and let us prove that if $K_i$ is sufficiently large, then it holds for $i$ with $K_{i}\geq K_{i+1}\geq \dots \geq K_k$. Again, for every $p_1,p_2\in S$ with $p_1\neq p_2$ colour the pair $\{p_1,p_2\}$ with blue if $\|p_1-p_2\|\geq t_i$ and with red otherwise. Let $B$ be a largest blue clique in $S$. Then $S$ can be partitioned into $|B|$ red cliques $R_1,\dots,R_{|B|}$ as in Lemma~\ref{clique}.

We may assume that $\phi(K_{i+1}) \le i$, otherwise we are done by induction. This implies that $\max_{i<j\leq k}\frac{t_{j}}{t_{j-1}+\varepsilon_k}\leq K_{i+1}$. Thus, by  Lemma~\ref{limit} (ii) we may assume that there exists a sufficiently small constant $\delta>0$ such that the following is \mbox{true for sufficiently small $\alpha_k$ and $\varepsilon_k$:}
\begin{equation}\label{eq413}\text{if }\min_{i<j\leq k}\frac{t_{j}}{t_{j-1}+\varepsilon_k}<1+\delta,\text{ then }|B|\leq m_{k-i}(d).\end{equation}

Set $K_i'=\max\left \{\frac{2}{\delta},K_{i+1}\right \}$. We are ready to verify the statement of the lemma for sufficiently large  $K_i>2K_i'$.
We separate two cases.

\medskip

\bf Case 1: \rm $\min_{i<j\leq k}\frac{t_{j}}{t_{j-1}+\varepsilon_k}< 1+\delta$. If $R$ is a largest red clique then we obtain
\begin{multline*}|S|\leq |B||R|\leq m_{k-i}(d)A_{i-1}(d',d)\leq \binom{k-i+d}{k-i}\binom{i-1+d}{i-1}< \binom{d+1}{k}\leq m_k(d),
\end{multline*}
where the second inequality follows from \eqref{eq413} and the fact that $R$ is an almost $(d,\alpha_k)$-flat $\varepsilon_k$-nearly $(i-1)$-distance set and that $\alpha_k,\varepsilon_k$ are sufficiently small, the third inequality follows by induction and by \eqref{eqbbs}, the fourth is true if $d$ is sufficiently large, and \mbox{the last is true by \eqref{eqlb}.}

\medskip

\bf Case 2: \rm $\min_{i<j\leq k}\frac{t_{j}}{t_{j-1}+\varepsilon_k}\geq 1+\delta$. Let $R$ be a largest red clique and let $R\cap B=\{b\}$.

We will apply Lemma~\ref{smallcliques} to $R\cup B$ with a sufficiently small $\gamma$ (to be chosen later) to bound $|R|$. Let $\beta'$ be as in Lemma~\ref{smallcliques} and
$0<\beta\leq \beta_0$ sufficiently small to be specified later. We may assume that $\alpha_k\leq \alpha(\beta)$ and choose $K_i$ such that $K_i\geq 2K(\alpha_k,\beta)$. Then Condition $1$ is satisfied automatically. We may further assume that $\varepsilon_k\le \beta^{d'+1}$. Then Condition $2$ is satisfied as well with $\beta$ if ``all distances from a point in $B$ to $R$ fall in one interval''. That is, to apply the Lemma, we need to as show that it is not possible to find indices $j_1>j_2\ge i$ and points $b, b'\in B$ with $b\neq b'$ and $r_1,r_2\in R$, such that $\|b'-r_1\|\in [t_{j_1},t_{j_1}+\varepsilon_k]$ and $\|b'-r_2\|\in [t_{j_2},t_{j_2}+\varepsilon_k]$. If that would have been the case, then, by the triangle inequality $t_{j_1}\le 
\|b'-r_1\|\le \|b'-r_2\|+\|r_1-r_2\|\le t_{j_2}+t_{i-1}+2\varepsilon_k,$ but, on the other hand, $t_{j_1}-t_{j_2}\ge \delta t_i\ge \delta K_it_{i-1}\ge \delta \frac 2{\delta}t_{i-1}\ge 2t_{i-1}>t_{i-1}+2\varepsilon_k,$ a contradiction. Thus, Condition $2$ is indeed satisfied as well.

Using Lemma~\ref{smallcliques}, we will show that $R$ is almost $(d-j,\gamma)$-flat. Since $S$ is almost $(d,\alpha_k)$-flat, we have that $R\cup B$ is $(r,d,\alpha_k)$-flat for all but at most two points $r\in R$. If $r\in R$ is such that $R\cup B$ is $(r,d,\alpha_k)$-flat, then by Lemma~\ref{smallcliques} we obtain that $R$ is $(r,d-j,\gamma)$-flat. Since this is true for all but at most two points $r\in R$, we obtain that $R$ is indeed almost $(d-j,\gamma)$-flat.

Thus, if $\gamma$ is sufficiently small, we have $|R|\leq A_{i-1}(d',d-j)$. Note also that $\max_{i< j\le k}\frac{t_{i}}{t_{i-1}+\varepsilon_k}\le K_{i+1}$, thus $B$ is $(j,20(K_{i+1}\beta)^{j/2})$-flat by Lemma \ref{global}. We obtain that $|B|\leq m_{k-i+1}(j)$ by Lemma~\ref{limit} (i) if  $\varepsilon_k$ and $\beta$ are sufficiently small.
Overall, for $|S|$ we obtain that
\[|S|\leq |R||B|\leq m_{k-i+1}(j)A_{i-1}(d',d-j),\]
if $\beta$, $\alpha_k$, and $\varepsilon_k$ is sufficiently small, and $K$ is sufficiently large. We separate two cases in order to bound $m_{k-i+1}(j)A_{i-1}(d',d-j)$.

\medskip

\bf Case 2.1: \rm $d-j\geq D_{i-1}$. In this case we obtain
\[
m_{k-i+1}(j)A_{i-1}(d',d-j)\leq  \binom{j+k-i+1}{k-i+1}\binom{d-j+i-1}{i-1}\leq \binom{d+1}{k}\leq m_k(d),\]
where the first inequality is true by induction and by \eqref{eqbbs}, the second is true if $d$ is sufficiently large, and the third is true by \eqref{eqlb}.

\medskip

\bf Case 2.2: \rm $d-j<D_{i-1}$. In this case we have $A_{i-1}(d',d-j)\leq 2+N_{i-1}(d',d-j)$. By Lemma~\ref{simple bound} we have $N_{i-1}(d',d-j)\leq C_i$ where $C_i$ depends only on $i$ and $D_{i-1}$, hence depends only on $i$. Thus, we obtain
\[m_{k-i+1}(j)A_{i-1}(d',d-j)\leq \binom{j+k-i+1}{k-i+1}(2+C_i)\leq\binom{d+1}{k}\leq  m_k(d),\]
where the first inequality follows from \eqref{eqbbs}, the second is true if $d$ is sufficiently large, and the third follows from \eqref{eqlb}.
\end{proof}

\subsection{Proof of Theorem~\ref{nearly k} (ii)}\label{sec23}
For $d=0$ the statement is obvious since $N_k(1)=M_k(0)=m_k(0)=1$ holds for any $k\geq 1$. For $d\geq 1$ in Claim \ref{A2} and in Claim\ref{A3} we will 
prove that for any $d'\geq d$ and $k=2,3$ we have $A_k(d',d) = m_k(d)$. This is sufficient in view of \eqref{eqrelation}.

\begin{cla}\label{A2}
We have $A_2(d',d)=m_2(d)$.
\end{cla}
\begin{proof}
Let $\varepsilon,\alpha>0$ be sufficiently small and $S$ be an almost $(d,\alpha)$-flat $\varepsilon$-nearly $2$-distance set in $\mathbb{R}^{d'}$ with distances $1\leq t_1<t_2$. Then for all but at most two points $p\in S$ we have that $S$ is $(p,d,\alpha)$-flat with respect to a $d$-dimensional subspace $\Gamma_p$. Let $K>3$ be a sufficiently large constant to be specified later. We may assume that $\frac{t_2}{t_1+\varepsilon}>K$, otherwise we have $|S|\leq m_2(d)$ by Lemma~\ref{limit} (i). For every $p_1,p_2\in S$ with $p_1\neq p_2$ colour the pair $\{p_1,p_2\}$ with blue if $\|p_1-p_2\|\geq t_2$ and with red otherwise. Let $B$ be a largest blue clique in $S$. Then $S$ can be partitioned into $|B|$ red cliques $R_1,\dots,R_{|B|}$ as in Lemma~\ref{clique}.

Let $\Gamma$ be the subspace spanned by $B-B=\setbuilder{b_1-b_2}{b_1,b_2\in B}$ and let $j$ be the dimension of $\Gamma$. 
Note that, since $B$ is an $\varepsilon$-nearly $1$-distance set, $B$ approximately a regular simplex. Hence, if $\varepsilon$ is sufficiently small, there is an absolute $\mu>0$ such that there is no $b\in B$ for which the set $B$ is $(b,j-1,\mu)$-flat. 

Let $R$ be any of the red cliques and let $B\cap R=\{b\}$. We will apply Lemma~\ref{smallcliques} to $R\cup B$ with a sufficiently small $\gamma$ (to be specified later). 
 
Let $\beta'$ be as in Lemma~\ref{smallcliques} and let $0<\beta\leq \beta'$ be sufficiently small. We may assume that $\alpha\leq \alpha(\beta)$ and that $K\geq K(\alpha_k,\beta)$. We may now apply Lemma~\ref{smallcliques}, as Conditions $1,2$ are automatically satisfied if $\varepsilon\le \beta^{d'+1}$.
If $j\geq 2$ and we have $\beta<\mu^{1/(j-1)}$, then $\beta^{j-1}<\mu$. Further, note that since we may assume that there is at least one blue edge, we have that $|B|\geq 2$. Thus, there is no $\gamma$ for which $B$ is $(0,\gamma)$-flat. In other words, there is no $\gamma$, for which $B$ is $(j-1,\gamma)$-flat if $j-1$. Thus, if $\beta$ is sufficiently small, then $j$ is as in Lemma~\ref{smallcliques}.

Since $S$ is $(p,d,\alpha)$-flat for all but at most two $p\in S$, for all but at most two (say $R_1$ and $R_2$) red cliques $R$ we have that $B\cup R$ is $(r,d,\alpha)$-flat for some $r\in R$. Thus by Lemma~\ref{smallcliques} we obtain that if $R$ is not $R_1$ or $R_2$, then $R$ is $(r,d-j,\gamma)$-flat for some $r\in R$. Now, since $R$ is an $\varepsilon$-nearly $1$-distance set, Lemma~\ref{limit} (i) implies that if $R$ is not $R_1$ or $R_2$, then we have $|R|\leq m_1(d-j)=d-j+1$ if $\varepsilon$ and $\gamma$ are sufficiently small. We also have $|R_1|+|R_2|\leq 2$, since there are only at most two $p\in S$ such that $S$ is not $(p,d,\alpha)$-flat.

Noting further that $|B|=j+1$, we obtain
\[|S|=|R_1|+\dots+|R_{|B|}|\leq \max \{(j+1)(d-j+1), j(d-j+1)+2\}.\]
Then either $d=j$ or $(j+1)(d-j+1)\geq  j(d-j+1)+2$ holds. In the first case, we have \[|S|\leq d+2\leq \binom{d+1}{2}\leq m_2(d)\] if $d\geq 3,$ and \[|S|\leq d+2\leq m_2(d)\] if $d=1,2$, since $m_2(1)=3$ and  $m_2(2)=5$. In the second case, we have \[|S|\leq (j+1)(d-j+1) \leq \left(\frac{d+2}{2}\right )^2\leq \binom{d+1}{2}\leq m_2(d)\] if $d\geq 4$, and \[|S|\leq (j+1)(d-j+1)\leq m_2(d)\] if $d=2,3$ since $m_2(2)=5$ and $m_2(3)=6$ (see Table \ref{small}).\\
\end{proof}

\begin{cla}\label{A3}
We have $A_3(d',d)=m_3(d)$.
\end{cla}

\begin{proof}
Let $\varepsilon,\alpha>0$ be sufficiently small and $S$ be an almost $(d,\alpha)$-flat $\varepsilon$-nearly $3$-distance set in $\mathbb{R}^{d'}$ with distances $0<t_1<t_2<t_3$. Later we will apply Lemma \ref{smallcliques} with a sufficiently small $\gamma$. 
Let $\beta'$ be as in Lemma \ref{smallcliques}, and let $0<\beta\leq \beta'$ be sufficiently small to be specified later. We may assume that $\alpha\leq \alpha(\beta)$ and that $K\geq \max\{3,K(\alpha,\beta)\}$ is sufficiently large. We may assume that $\frac{t_2}{t_1+\varepsilon}\geq K$ or $\frac{t_3}{t_2+\varepsilon}\geq K$ holds, otherwise we immediately obtain $|S|\leq m_3(d)$ by Lemma~\ref{limit} (i) if $\varepsilon$ and $\alpha$ are sufficiently small. We will analyse these two cases separately.

\medskip

\bf Case 1: \rm $\frac{t_3}{t_2+\varepsilon}\geq K$. For every $p_1,p_2\in S$ with $p_1\neq p_2$ colour the pair $\{p_1,p_2\}$  with blue if $\|p_1-p_2\|\geq t_3$ and with red otherwise. Let $B$ be a largest blue clique in $S$. Then $S$ can be partitioned into $|B|$ red cliques $R_1,\dots,R_{|B|}$ as in Lemma~\ref{clique}.

Let $\Gamma$ be the subspace spanned by $B-B=\setbuilder{b_1-b_2}{b_1,b_2\in B}$, and let $j$ be the dimension of $\Gamma$. 
Note that, since $B$ is an $\varepsilon$-nearly $1$-distance set, $B$ is approximately a regular simplex. Hence, there is an absolute $\mu>0$ such that if $\varepsilon$ is sufficiently small, then for no $b\in B$, the set $B$ is $(b,j-1,\mu)$-flat. 

Let $R$ be any of the red cliques and let $B\cap R=\{b\}$. We will apply Lemma~\ref{smallcliques} to $R\cup B$ with $\gamma$. Conditions $1,2$  of Lemma \ref{smallcliques} are automatically satisfied if $\varepsilon\le \beta^{d'+1}$. Moreover, if $j\geq 2$ and we have $\beta<\mu$, then $\beta^{j-1}<\mu$, thus $j$ is as in Lemma~\ref{smallcliques}. (The $j=1$ case can be handled in the same way as in Claim \ref{A2}.)

Since $S$ is $(p,d,\alpha)$-flat for all but at most two $p\in S$, for every red clique $R$, for all but at most two points $r\in R$, we have that $B\cup R$ is $(r,d,\alpha)$-flat. Thus, by Lemma~\ref{smallcliques} we obtain that every $R$ is $(r,d-j,\gamma)$-flat for all but at most two $r\in R$. Moreover, since $R$ is an $\varepsilon$-nearly $2$-distance set, we obtain that $|R|\leq A_2(d',d-j)$ if $\gamma$ is sufficiently small. Noting further that $|B|=j+1$, overall we obtain
\[|S|=|R_1|+\dots+|R_{|B|}|\leq (j+1)A_2(d',d-j).\]
For sufficiently small $\gamma$ and any red clique $R$, we have $|R|\leq 2$ if $d=j$. In this case it follows that $|S|\leq 2(d+1)$. Then for $d\geq 4$ we have \[|S|\leq 2(d+1)\leq \binom{d+1}{3}\leq m_3(d),\]
where the second inequality is true by a simple calculation, and the third is by \eqref{eqbbs}.
For $d=1,2,3$ we have \[|S|\leq 2(d+1)\leq m_3(d)\] given that $m_3(1)=4$, $m_3(2)=7$ and $m_3(3)=12$ (see Table \ref{small}).

If $j<d$, then $|R|\leq A_2(d',d-j')\leq m_2(d-j)$, where the second inequality is by the $k=2$ case of the theorem (Claim \ref{A2}) for sufficiently small $\gamma$. In this case, for $d\geq 9$ we have \[|S|\leq (j+1)m_2(d-j)\leq (j+1)\binom{d-j+2}{2}\leq \binom{d+1}{3}\leq m_3(d),\]
where the second inequality is true by \eqref{eqbbs}, the third by a simple calculation, and the fourth by \eqref{eqlb}.
For $d\leq 8,$ using the known values and bounds of $m_2(d)$ and $m_3(d),$ we check in the Appendix that \begin{equation}\label{eqcheck}(j+1)m_2(d-j)\leq m_3(d).\end{equation}

\medskip

\bf Case 2: \rm $\frac{t_2}{t_1+\varepsilon}\geq K>\frac {t_3}{t_2+\varepsilon}$. For every $p_1,p_2\in S$ with $p_1\neq p_2$ colour the pair $\{p_1,p_2\}$ with blue if $\|p_1-p_2\|\geq t_2$ and with red otherwise. Let $B$ be a largest blue clique in $S$. Using Lemma~\ref{clique}, partition the set $S$ into $|B|$ red cliques $R_1,\dots,R_{|B|}$. We split the analysis into two more subcases.

\medskip

\bf Case 2.1: \rm $\frac{t_3}{t_2+\varepsilon}> 1+\frac{2}{K}$. Let $R$ be one of the red cliques and let $R\cap B=\{b\}$. We will apply Lemma~\ref{smallcliques} to $R\cup B$ with $\gamma$, similarly as before. Condition $1$ is automatically satisfied. To check Condition 2, note that for any $p_1,p_2\in R$ and $b'\in B$ with $b\neq b'$, if $|p_1-b'|\in [t_i,t_i+\varepsilon]$ and $|p_2-b'|\in [t_{\ell}, t_{\ell}+\varepsilon]$, then by the triangle inequality we obtain that $\ell=i$ if $\varepsilon$ is sufficiently small. Thus, Condition 2. is satisfied as well if $\varepsilon\le \beta^{d'+1}$ and $\varepsilon$ is sufficiently small.

Let $j$ be as in Lemma \ref{smallcliques}. If for some $p\in R$ we have that $S$ is $(p,d,\alpha)$-flat, then $R$ is $(p,d-j,\gamma)$-flat by Lemma~\ref{smallcliques}. Further, note that the same is true for any red clique $R$ with the same $j$. Indeed, since $\frac{t_3}{t_2+\varepsilon}<K$,  Lemma~\ref{global} implies that if $\beta$ is sufficiently small, then there is a $j$ such that $B$ is $(b',j,\beta^j)$-flat for every $b'\in B$, but there is no $b'\in B$ for which it is $(b',j-1,\beta^{j-1})$-flat. This and Lemma~\ref{limit} (i) imply that for all but at most two red cliques $R$ we have $|R|\leq m_1(d-j)=d-j+1$ if $\gamma$ and $\varepsilon$ are sufficiently small. Moreover, if the two potential exceptions are say $R_1$, $R_2$, then $|R_1|+|R_2|\leq 2$. Note also that by Lemma \ref{limit} (i) we have $|B|\leq m_2(j)$ if $\beta$ and $\varepsilon$ are sufficiently small.

Overall, we obtain
\[|S|\leq |R_1|+\dots+|R_{|B|}|\leq \max \{m_2(j)(d-j+1), (m_2(j)-1)(d-j+1)+2\}.\]
Then we either have $d=j$ or $j\leq d-1$, and thus $m_2(j)(d-j+1)\geq (m_2(j)-1)(d-j+1)+2$. In the first case ($d=j$) for $d\geq 6$ we have \[|S|\leq m_2(d)+1\leq \binom{d+2}{2}+1\leq \binom{d+1}{3}\leq m_3(d),\]
where the second inequality is by \eqref{eqbbs}, the third is by a simple calculation, and the fourth is by \eqref{eqlb}. 
For $1\leq d \leq 5$ we have \[|S|\leq m_2(d)+1\leq m_3(d)\]
since $m_2(1)=3$, $m_2(2)=5$, $m_2(3)=6$, $m_2(4)=10$, $m_2(5)=16$ and $m_3(1)=4$, $m_3(2)=7$, $m_3(3)=12$, $m_3(4)=16$, $m_3(5)\geq 24$ (see Table \ref{small}). Finally, in the second case ($j\leq d-1$) we do the same analysis as in the end of Case 1.\\

\vskip+0.1cm

\bf Case 2.2: \rm $\frac{t_3}{t_2+\varepsilon}\leq 1+\frac{2}{K}.$
First, we will show that $|B|\leq d+1$. Indeed, we either have that $|B|\leq 2$, or there is a $b\in B$ such that $B$ is $(b,d,\alpha)$-flat. In the latter case, by Lemma \ref{global} we obtain that $B$ is $(d,20((1+2/K)\alpha)^{1/2})$-flat. Then, by Lemma~\ref{limit} (ii), if $\frac{2}{K}$, $\alpha$ and $\varepsilon$ are sufficiently small, we have that $|B|\leq m_1(d)= d+1$.

Next, we will show that for any red clique $R$ we have $|R|\leq d+1$. Indeed, we either have that $|R|\leq 2$, or there is an $r\in R$ such that $R$ is $(r,d,\alpha)$-flat. In the latter case, by Lemma \ref{global} we obtain that $R$ is $(d,20(\alpha)^{1/2})$-flat. Then, by Lemma~\ref{limit} (i), if $\alpha$ and $\varepsilon$ are sufficiently small, we have that $|R|\leq m_1(d)= d+1$.

We obtain that \[|S|=|R_1|+\dots+|R_{|B|}|\leq (d+1)^2.\] Then if $d\geq 9$, it follows by a simple calculation and by \eqref{eqlb} that \[|S|\geq (d+1)^2\leq \binom{d+1}{3}\leq m_3(d).\]
Further, for $d=7$ and for $d=8$ we have $m_3(8)\geq 121\geq (8+1)^2$, $m_3(7)\geq 65\geq (7+1)^2$ (see Table \ref{small}.) Therefore, in the rest of the proof we may assume that $d\leq 6$.

\medskip

\bf Case 2.2.1: $t_1\ge K^{0.1}(t_3-t_2)$.
\rm 
\begin{figure}
\begin{center}
\begin{tikzpicture}
\draw[thick,blue] (-20:6) arc (-20:20:6);
\draw[thick,blue] (-20:6.3) arc (-20:20:6.3);
\draw[thick,dashed] (-2,0) -- (7,0);
\draw[very thick,red] (6,0) -- (6.3,0);
\draw[->,thick,red] (6.8,0.7) -- (6.15,0.05);

\filldraw (-18:6) circle (2pt);
\filldraw (11:6.2) circle (2pt);
\draw[thick,teal] (-18:6.0) -- (11:6.2);
\filldraw[blue] (0,0) circle (2pt);

\node[fill=white] at (-18:5.6) {$r$};
\node[fill=white] at (11:6.65) {$q$};
\node[fill=white] at (0,0.4) {$b'$};
\node[fill=white] at (20:5.65) {\textcolor{blue}{$\mathcal S_2$}};
\node[fill=white] at (18:6.75) {\textcolor{blue}{$\mathcal S_3$}};
\node[fill=white] at (7.2,0.5) {\textcolor{red}{$t_3-t_2$}};
\node[fill=white] at (5.53,-0.35) {\textcolor{teal}{$t_1$}};

\end{tikzpicture}
\caption{}
\end{center}
\end{figure}
Let $R$ be a largest red clique. To bound the cardinality of $R$ in this case we will not use Lemma~\ref{smallcliques}, but we will use Lemma \ref{flatslemma} directly.

Let $R\cap B=\{b\}$, and let $\gamma$ be sufficiently small (to be specified later). For $\gamma'=\frac{\gamma}{2}$ let $\beta_0$ as in Lemma \ref{flatslemma}, and let $\beta\leq \beta_0$ be sufficiently small. Further, let $r$ be any point of $R$, and let $V=\setbuilder{b'-r}{b'\in B\setminus \{b\}}$. Assume that $j\geq 1$ is the smallest integer such that $V$ is $(j,\beta^j)$-Flat, and assume that $V$ is $(j,\beta^j)$-flat with respect to $\Gamma$. Let $q\in R$ be any point of $R$ such that $q\neq r$.

\begin{cla}\label{cond2} For every $b'\in B\setminus \{b\}$ the angle between the vectors $(q-r)$ and $(b'-r)$ falls in $[\frac{\pi}{2}-\frac{\beta^{d'}}{2},\frac{\pi}{2}+\frac{\beta^{d'}}{2}]$, if $K$ is sufficiently large, and $\varepsilon$ is sufficiently small.
\end{cla}

\begin{proof}
Let $\mathcal S_2$, $\mathcal S_3$ be spheres centred at $b'$ and of radii $t_2$ and $t_3$ respectively (see Figure 2). Then $r$ is $\varepsilon$-close to one of them. We only spell out the proof in the case when $r$ is $\varepsilon$-close to $S_1$, as the case when it is $\varepsilon$-close to $S_3$ can be done very similarly. Note that $q$ is also $\varepsilon$-close to $\mathcal S_2$ or $\mathcal S_3$. If $q$ is $\varepsilon$-close to $\mathcal S_2$, then some simple calculation shows that for some absolute constant $c_1$ the angle between the vectors $(q-r)$ and $(b'-r)$ falls in $[\frac{\pi}{2}-c_1/K,\frac{\pi}{2}+c_1/K]$. If $q$ is $\varepsilon$-close to $\mathcal S_3$, then we claim that for some absolute constant $c_2$ the angle between the vectors $(q-r)$ and $(b'-r)$ falls in $[\frac{\pi}{2}-c_2/K^{0.1},\frac{\pi}{2}+c_2/K^{0.1}]$. Indeed, this follows from the facts that $|q-r|\in[t_1,t_1+\varepsilon]$, $t_1\ge K^{0.1}(t_3-t_2)$, and that the radius of $\mathcal S_3$ is much bigger than $t_1$. Thus, we can conclude that if $K$ is sufficiently large, then the angle between the vectors $(q-r)$ and $(b'-r)$ falls in  $[\frac{\pi}{2}-\frac{\beta^{d'}}{2},\frac{\pi}{2}+\frac{\beta^{d'}}{2}]$
\end{proof}

Next, we will show that if for some $r\in R$ the set $R\cup B$ is $(r,d,\alpha)$-flat, then $R$ is $(r,d-j,\gamma)$-flat. This part of the proof is very similar to the proof of Lemma \ref{flatslemma}, but for completeness we spell it out with all details.

Assume that $B\cup R$ is $(r,d,\alpha)$-flat with respect to a $d$-dimensional subspace $\Gamma_r$. Let $\pi_r(V)$ be the projection of $V=\setbuilder{b'-r}{b'\in B\setminus \{b\}}$ on $\Gamma_r$. Assume that $j_r'\geq 1$ is the lowest dimension such that $\pi_r(V)$ is $(j_r,\frac{\beta^{j_r}}{2})$-Flat. Further, let $\Gamma_r'$ be a subspace of dimension $j_r$ such that $\pi_r(V)$ is $(j_r,\frac{\beta^{j_r}}{2})$-Flat with respect to $\Gamma_r'$.

The angle between $(b'-b)$ and $(b'-r)$ is at most $\alpha$ if $K$ is sufficiently large and $\beta$ is sufficiently small. The angle between $(b'-r)$ and $\pi_r(b'-r)$
is at most $\alpha$. Further, the angle between $\pi_r(b'-r)$ and $\Gamma_r$ is at most $\frac{\beta^{j_r}}{2}$. These, together with the triangle inequality imply that $B$ is $(b,j_r,\beta^{j_r})$-flat with respect to $\Gamma'_r$ if $2\alpha\leq \frac{\beta^{j_r}}{2}$. By the minimality of $j$, it follows than that $j_r\geq j$.

Let $\Lambda_r$ be the affine plane through $r$ parallel to $\Gamma'_r$. For a point $p\in \mathbb{R}^{d'}$ we denote by $\pi'_r(p)$ the projection of $p$ on $\Lambda_r$. Similarly, for a set $X\subseteq \mathbb{R}^{d'}$, let $\pi'_r(X)$ denote the projection of $X$ on $\Lambda_r$. Note that $\pi'_r(r)=r$. Let $(b'-r)^{\bot_r}$ denote the $(d-1)$-dimensional subspace in $\Gamma'_r$ that is orthogonal to $b'-r$.
\begin{cla}
For any $b'\in B\setminus \{b\}$ and any $r\in R$, the projection $\pi'_r(R)$ is $(r,d-1, \beta^{d})$-flat with respect to $(b'-r)^{\bot_r}$ if $\alpha$ and $\beta$ are sufficiently small and $K$ is sufficiently large. 
\end{cla}
\begin{proof}
Let $q\in R\setminus \{r\}$ any point of $R$. It follows from Claim \ref{cond2} that if $K$ is sufficiently large and $\varepsilon$ is sufficiently small, then the angle between the vectors $(q-r)$ and $(b'-r)$ falls in $[\frac{\pi}{2}-\frac{\beta^{d'}}{2},\frac{\pi}{2}+\frac{\beta^{d'}}{2}]$. Further, we have $\angle qr\pi'_r(q)\leq \alpha$. Thus, $\angle b'r\pi'_r(q)\in [\frac{\pi}{2}-\beta^{d},\frac{\pi}{2}+\beta^{d}]$ if $\alpha$ is sufficiently small.
Since $\pi'_r(R)$ is contained in $\Lambda_r$, we obtain that for every $b'\in B$ the set $\pi'_r(R)$ is $(r,d-1,\beta^{d})$-flat with respect to $(b'-r)^{\bot_r}$.
\end{proof}

Now we apply Lemma \ref{flatslemma} with $\pi_r(V)$ and $\gamma'=\frac{\gamma}{2}$, and obtain that $\pi'_r(R)$ is $(r,d-j_r,\frac{\gamma}{2})$-flat. This, by $j_r\geq j$, implies that $\pi_r(R)$ is $(r,d-j,\frac{\gamma}{2})$-flat. Since for any $q\in R$ the angle between $(q-r)$ and $(\pi_r(q)-r)$ is at most $\alpha$, it follows that $R$ is $(r,d-j,\gamma)$-flat if $\alpha$ is sufficiently small.

Thus, either there is no $r\in R$ for which $R$ is $(r,d,\alpha)$-flat, in which case $|R|\leq 2$, or there is an $r\in R$ such that $R$ is $(r,d-j,\gamma)$-flat. In the latter case, by Lemma \ref{global} we obtain that $R$ is $(d-j,20(\gamma)^{1/2})$-flat. Thus, by Lemma~\ref{limit} (i) we have $|R|\leq m_1(d-j)$ if $\gamma$ and $\varepsilon$ are sufficiently small. Overall, we have that $|R|\leq \max\{2,m_1(d-j)\}$. 

Finally, we claim that $|B|\leq j+1$. To see this, note if $K$ is sufficiently large, then the angle between the vectors $(b'-r)$ and $(b-r)$ is at most $\beta^{j}$. Then, since $V$ is $(j,\beta^j)$-Flat, we obtain that $B$ is $(b,j,2\beta^j)$-flat. Lemma \ref{global} implies that $B$ is $(j,20((1+2/K)\beta^j)^{1/2})$-flat. By Lemma \ref{limit} (ii) we conclude that $|B|\leq m_1(j)=j+1$ if $\beta$ and $\varepsilon$ are sufficiently small.  

Overall, we obtain that \[|S|\le (j+1)(d-j+1)\le m_3(d),\] where the second inequality was already proven in the previous cases.\\

\bigskip

\bf Case 2.2.2: \rm
$t_1\le K^{0.1}(t_3-t_2)$. For every $1\leq i\leq |B|$ let $\{b_i\}=B\cap R_i$, and for $m=2,3$, let $\mathcal S_{\ell}(i)$ be the sphere of radius $t_{m}$ centred at $b_i$. We need the following claim.

\begin{cla}\label{cla23}Assume that for some $1\leq i, \ell \leq |B|$ with $i\neq \ell$ there are points from $R_i$ in the $\varepsilon$-neighbourhoods of both $\mathcal S_2(\ell)$ and $\mathcal S_3(\ell)$. Then $R_{\ell}$ is contained in the $\varepsilon$-neighbourhood either of $\mathcal S_2(i)$ or of $\mathcal S_3(i)$.
\end{cla}

\begin{proof}Assume the contrary. We may assume  that $|b_i-b_{\ell}|\in [t_2,t_2+\varepsilon]$ (the case with $t_2$ replaced by $t_3$ can be treated similarly). Then there are points $p\in R_i$, $q\in R_{\ell}$ such that $p$ is in the $\varepsilon$-neighbourhood of $\mathcal S_3(\ell)$, and $q$ is in the $\varepsilon$-neighbourhood of $\mathcal S_3(i)$ (see Figure 3). Let $p',q'$ denote the projections of $p,q$ on the line $e$ passing through $b_i$ and $b_{\ell}$, and let $r_i$ and $r_{\ell}$ denote the points of intersection of $e$ and spheres $\mathcal S_3(\ell),$ $\mathcal S_3(i)$ respectively. Note that $\|r_i-r_{\ell}\|\ge t_3+(t_3-t_2).$ 

\begin{figure}
\begin{center}
\begin{tikzpicture}
\node[fill=white] at (-0.7,1.1) {$p$};
\node[fill=white] at (-0.8,0.3) {$p'$};
\node[fill=white] at (11:6.75) {$q$};
\node[fill=white] at (6.75,0.3) {$q'$};

\node[fill=white] at (0.3,0.3) {\textcolor{blue}{$b_i$}};

\node[fill=white] at (5.7,0.3) {\textcolor{blue}{$b_{\ell}$}};

\node[fill=white] at (3,0.3) {$e$};

\node[fill=white] at (20:5.35) {\textcolor{blue}{$\mathcal S_2(i)$}};
\node[fill=white] at (18:7.15) {\textcolor{blue}{$\mathcal S_3(i)$}};
\node[fill=white] at (0.9,1.75) {\textcolor{blue}{$\mathcal S_2(\ell)$}};
\node[fill=white] at (-0.8,2) {\textcolor{blue}{$\mathcal S_3(\ell)$}};

\draw[thick,blue] (-20:6) arc (-20:20:6);
\draw[thick,blue] (-20:6.5) arc (-20:20:6.5);
\draw[thick,blue] (0,0) arc (180:200:6);
\draw[thick,blue] (0,0) arc (180:160:6);
\draw[thick,blue] (-0.5,0) arc (180:200:6.5);
\draw[thick,blue] (-0.5,0) arc (180:160:6.5);

\draw[thick,dashed] (-0.36,1) -- (-0.36,0);
\draw[thick,dashed] (11:6.45) -- (6.36,0);

\draw[thick,dashed] (-1,0) -- (7,0);

\filldraw (-0.36,1) circle (2pt);
\filldraw (6.36,0) circle (2pt);
\filldraw (11:6.45) circle (2pt);
\filldraw[fill=blue] (6,0) circle (2pt);
\filldraw (-0.36,0) circle (2pt);
\filldraw[blue] (0,0) circle (2pt);

\end{tikzpicture}
\caption{}
\end{center}
\end{figure}

We claim that $\|r_i-p'\|,\|r_{\ell}-q'\|\le (t_3-t_2)/10.$ This would imply that $\|q-p\|\ge \|q'-p'\|\ge \|r_i-r_{\ell}\| -\frac 2{10}(t_3-t_2)>t_3+\varepsilon$, which is a contradiction. Let us only show that $\|r_{\ell}-q'\|\le (t_3-t_2)/10,$ since the other inequality can be proven in the same way.  
Due to our condition on $t_2$, we have $\|r_{\ell}-q\|\le \|r_{\ell}-b_{\ell}\|+\|b_{\ell}-q\|\le 2t_i+2\ep \le 3K^{0.1}(t_3-t_2).$ Since we have $t_3-t_2\le 2t_3/K,$ and $q$ lies in the $\varepsilon$-neighbourhood of $\mathcal S_3(i),$ for the angle $\gamma$ between the vector $(q-r_{\ell})$ and the line $e$ we have $2\cos \gamma = \frac {\|r_{\ell}-q\|}{\|q-b_i\|}\le \frac {3K^{0.1}(t_3-t_2)}{t_3}\le 3K^{-0.9}.$ Therefore, we have $\|r_{\ell}-q'\| = \|r_{\ell}-q\|\cos \gamma \le 3/K^{0.8}<(t_3-t_2)/10$ for sufficiently large $K$. 
\end{proof}

Assign an ordered pair $(\rho_1,\rho_2)$ to each ordered pair $(i,\ell)$ with $i\neq \ell,$ if $R_{i}$ can be covered by the $\varepsilon$-neighbourhood of $\rho_1$ many spheres out of $\mathcal S_2(\ell),\mathcal S_3(\ell)$, and $R_{\ell}$ can be covered by the $\varepsilon$-neighbourhood of $\rho_2$ many spheres out of $\mathcal S_2(i),\mathcal S_3(i)$. By Claim~\ref{cla23} we have that $(\rho_1,\rho_2)\in \{(1,1),(2,1),(1,2)\}.$ If there are $\tau(i)$ indices $\ell^1,\ldots,\ell^{\tau(i)}\in B\setminus \{i\}$ such that we assigned $(1,2)$ or $(1,1)$ to $(i,\ell)$, then $R_{i}$ is contained in the intersection of the $\varepsilon$-neighbourhood of $\tau(i)$ spheres of radii $t_2$ or $t_3$ (with centres in $b_{\ell^1},\ldots, b_{\ell^{\tau(i)}}$). 

Let $\gamma$ be sufficiently small, that will be specified later. Using Lemma \ref{flatslemma}, we will show that for any $r\in R_{i}$ the set $R_i$ is $(r,d-\tau(i),\gamma)$-flat provided that $\varepsilon$, $\alpha$ are sufficiently small, and $K$ is sufficiently large. Let $\beta_0$ be as in Lemma \ref{flatslemma} for $\gamma'=\frac{\gamma}{2}$, and let $\beta\leq \beta_0$ be sufficiently small.

We denote by $\Gamma_{i}'$ the subspace spanned by the set of vectors $\setbuilder{b_{\ell^s}-b_i}{1\leq s\leq \tau(i)}$. Standard calculation shows that if $\varepsilon$ and $\beta$ are sufficiently small, and $K$ is sufficiently large, then there is no $j<\tau(i)$ for which $\Gamma_{i}'$ is  $(j,\beta^j)$-Flat . (On an intuitive level, this is the case because $B$ is approximately a regular simplex.)
On the other hand, $\Gamma_{i}'$ is of dimension at most $\tau(i)$, thus it is $(\tau(i),\beta^{\tau(i)})$-Flat.

Let $p\in R_i$ be such that $p\neq r$. If $\varepsilon$ is sufficiently small and $K$ is sufficiently large, then for any $1\leq s\leq \tau(i)$ the angle of the vectors $(b_{\ell^s}-r)$ and $(r-p)$ is $\frac{\beta^{d'}}{2}$-close to $\frac{\pi}{2}$. Indeed, this follows since the lengths $\|b_{\ell^s}-r\|$ and $\|b_{\ell^s}-p\|$ are $\varepsilon$-close to each other, and $\|b_{\ell^s}-r\|$ is at least $K$ times as large as $\|r-p\|$. Further, if $\varepsilon$ is sufficiently small and $K$ is sufficiently large, then the angle of the vectors $(b_{\ell^s}-b_i)$ and $(b_{\ell^s}-r)$ is at most $\frac{\beta^{d'}}{2}$. Thus, the angle of $(r-p)$ and $(b_{\ell^s}-b_i)$ is $\beta^{d'}$-close to $\frac{\pi}{2}$. 

Assume that for some $r\in R_{i}$ the set $R_i\cup B$ is $(r,d,\alpha)$-flat with respect to $\Gamma_r$. Then using Lemma \ref{flatslemma} after projecting to a subspace parallel to $\Gamma_r$ through $r$, we obtain that $R_i$ is $(r,d-\tau(i),\alpha)$-flat. We omitted the details of this argument, as they are essentially the same as the proof of Lemma \ref{smallcliques} and of the proof of Case 2.2.1 after Claim \ref{cond2}.

Recall that $S$ is $(r,d,\alpha)$-flat for all but at most two $r\in S$. Thus, for all but at most two (say, $R_{1}$ or $R_{1}, R_{2}$) sets $R_{i}$ there is an $r\in R_i$ such that $R_i$ is $(r,d-\tau(i),\gamma)$-flat. Then Lemma \ref{global} and Lemma~\ref{limit} (i) together imply that $|R_{i}|\leq d-\tau(i)+1$ if $\alpha$ if $\gamma$ is sufficiently small.

Each pair of vertices contributes to at least $1$ of the $\tau(i)$'s, which implies that \[\sum_{i=1}^{|B|} \tau(i) \ge \binom{|B|}{2}.\]
If in all $R_{i}$ there is an $r$ such that $R$ is $(r,d,\alpha)$-flat, then we obtain 
\begin{equation}\label{first}|S|=\sum_{i=1}^{|B|}|R_i|\leq |B|(d+1)-\sum_{i=1}^{|B|} \tau(i)\le|B|(d+1)- \binom{|B|}{2}.
\end{equation}
Otherwise, repeating the same argument for $S':=\bigcup_{i=2}^{|B|} R_i$ or for $S'':=\bigcup _{i=3}^{|B|}R_i$ ,and using $|R_1|\le 2$ or $|R_{1}|+|R_{2}|\leq 2$, we obtain 
\begin{equation}\label{second}
|S|=\sum_{i=1}^{|B|}|R_i|\leq (|B|-1)(d+1)-\binom{|B|-1}{2}+2.
\end{equation}

By Lemma \eqref{limit} (ii) we have $|B|\leq d+1$, if $K$ is sufficiently large and $\varepsilon$ and $\alpha$ are sufficiently small. Thus, recalling that we assumed $d\leq 6$, in both \eqref{first} and \eqref{second} the right hand side is bounded from above by $m_3(d)$, by a simple calculation and by the fact that  $m_3(2)=7$, $m_3(3)=12$,  $m_3(4)=16$, $m_3(5)\geq 24$, $m_3(6)\geq 40$ (see Table \ref{small}).
\end{proof}

\subsection{Proof of Theorem~\ref{fixd}}\label{sec24}
In order to prove the theorem, we will need the spherical analogues of our quantities. For a set $P$ on a $d$-sphere $\mathcal S^d\subset \mathbb R^{d'}$ centred at $\mathbf{0}$, we say that $P$ is {\it  $(d,\alpha,\mathcal S^d)$-flat} if for each $p\in P$ it is $(p,d,\alpha)$-flat with respect to a $d$-dimensional subspace $\Gamma_p$ that contains the vector $p-\mathbf 0$. Note that we do not impose any conditions on the radius of the sphere.

Let $NS_k(d',d)$ denote the largest number $M$ such that for any $\alpha, \varepsilon>0$ there is a $(d,\alpha,\mathcal S^d)$-flat $\varepsilon$-nearly $k$-distance set of cardinality $M$ on a $d$-sphere $\mathcal{S}^{d}\subset \mathbb{R}^{d'}$. 

We call two subspaces $\Gamma_1$ and $\Gamma_2$ of $\mathbb{R}^d$ \emph{intersecting-orthogonal} if there is an orthogonal basis $\{v_1,\dots,v_d\}$ of $\mathbb{R}^d$ with indices $1\leq i \leq j \leq d$ such that $\{v_1,v_2\dots,v_j\}$ is an orthogonal basis of $\Gamma_1$ and $\{v_i,v_{i+1}\dots,v_d\}$ is an orthogonal basis of $\Gamma_2$. Slightly abusing notation, we will also call two affine planes $\Lambda_1$ and $\Lambda_2$ of $\mathbb{R}^d$ \emph{intersecting orthogonal} if the subspaces $\Gamma_1=\Lambda_1-\Lambda_1$ and $\Gamma_2=\Lambda_2-\Lambda_2$ are intersecting orthogonal.

For each $d,d',$ with $0\le d<d'$, we are going to prove that $N_k(d',d),NS_k(d',d) \le 2(k+1)^d$ simultaneously by induction on $d$. As the proof for $N_k(d',d)$ and for $NS_k(d',d)$ are very similar, we only spell it out with details for $NS_k(d',d)$, which is the slightly more complicated case. Then we will comment on how to modify the proof for $N_k(d',d)$.

The statement for $NS_k(d',d)$ is clear for $d=0$ and for any $d'$. Suppose that the statement holds for $d-1$. More precisely, we assume that there exist  $\varepsilon_{d-1},\alpha_{d-1}>0$ such that any  $(d-1,\alpha_{d-1},\mathcal S^{d-1})$-flat $\varepsilon_{d-1}$-nearly $k$-distance set $P$ on a $(d-1)$-sphere $\mathcal{S}^{d-1}\subset \mathbb{R}^{d'}$ satisfies $|P|\le 2(k+1)^{d-1}$. 
We are now going to prove a similar statement for $d$
with $\varepsilon_{d},\alpha_d>0$, where $\varepsilon_d,\alpha_d$ are sufficiently small compared to $\varepsilon_{d-1},\alpha_{d-1}$.

Fix some sufficiently small $\varepsilon_d,\alpha_d>0$. Take a $(d,\alpha_d,\mathcal S^d)$-flat $\varepsilon_d$-nearly $k$-distance set $P$ of points on a sphere $\mathcal S^d$. Let $\rho$ be the radius of $\mathcal{S}^d$, and let the $k$ distances be $1\leq t_1\leq \dots \leq t_k$. Note that we may also assume that $t_1\geq 2$.
Indeed, to get this, simply enlarge $P$ from $\mathbf{0}$. Then the enlarged image is a $(d,\alpha_d,\mathcal{S}^d)$-flat $2\varepsilon_d$-nearly $k$-distance set with distances $2t_1\leq \dots \leq 2t_k$.

Take any point $p\in P$, and for each $i\in[k]$ let $\mathcal{S}_i^{d-1}$ be the $(d-1)$-sphere obtained as the intersection of $\mathcal{S}^d$ with the $d$-sphere $\mathcal S(p,t_i)$ of radius $t_i$ centred at $p$.  Note that every point of $q\in P\setminus \{p\}$ is contained in the $\varepsilon_d$-neighbourhood of  $\mathcal{S}(p,t_i)$ for some $i$. 
Further, let $\Gamma_p$ be a subspace of dimension $d$ containing the vector $p-\mathbf 0$ such that $P$ is $(p,d,\alpha_d)$-flat with respect to $\Gamma_p$.

Let $j'$ be the largest index $j$ such that $t_{j}<\varepsilon^{1/4}_d\rho$, if there is any, and otherwise let $j'=0$. 
Then all points at distance at most $t_{j'}+\varepsilon_d$ from $p$ lie in a spherical cap with centre in $p$ and of angular radius $K_1\varepsilon^{1/2}_d$ for some constant $K_1$.
Denote the set of these points by $X$. For every $q\in X$, let $\Lambda_q$ be the $d$-dimensional affine plane (contained in the $(d+1)$-dimensional affine plane spanned by $\mathcal{S}^d$) tangent to $\mathcal S^d$ at $q$.  Further, let $\Gamma'_q$ be a $d$-dimensional subspace parallel to $\Lambda_q$. Then one can show (by combining a projection argument with standard calculations) that there is a universal constant $K_2$ such that $X$ is globally $(d,K_2\varepsilon^{1/2}_d)$-flat with respect to  $\Gamma_q'$.

Recall that at the same time for every $q\in X$ we have that $X$ is $(q,d,\alpha_d)$-flat with respect to a $d$-dimensional subspace $\Gamma_q$ containing $q-\mathbf{0}$. Since $\Gamma_q$ contains the vector $q-\mathbf{0}$, the subspace $\Gamma_q'$ is intersecting-orthogonal to $\Gamma_q$. Let $\Gamma_q''$ denote the intersection of $\Gamma_q$ and $\Gamma_q'$. We can conclude by simple calculation that for any $q\in X$ we have that $X$ is $(q,d-1,\alpha_{d-1})$-flat with respect to $\Gamma_q''$ provided $\varepsilon_d,\alpha_d$ are chosen appropriately small. Thus, by the induction hypothesis we obtain $|X|\le 2(k+1)^{d-1}$.

Next, let $j''$ be the smallest index such that  $t_{j''}\ge  (2-\varepsilon^{1/2}_d)\rho$
(if there is no such $j''$ then we put $j'':=k+1$). 
Let $Y$ be the set of those points of $P$ that are at distance at least $t_{j''}$ from $p$. Let $p'$ be the point on $\mathcal{S}^d$ that is diametrically opposite to $p$. Then $Y$ is contained in a spherical cap with centre in $p'$ and of angular radius $K_1\varepsilon^{1/2}_d$. Following the same argument that we used to bound $|X|$, we obtain that $|Y|\le 2(k+1)^{d-1}$.

Now consider the set $P':=P\setminus (X\cup Y\cup\{p\})$. For every point of $q\in P'$ there is an index $j'<i<j''$ such that the distance of $q$ and $p$ falls in $[t_i,t_i+\varepsilon_d]$. Note also that for every point $q\in P'$ the angle between $q-\mathbf {0}$ and $p-\mathbf 0$ is at least $\varepsilon^{1/2}_d$. Standard calculations show that if for some $q\in P'$ and $j'<i<j''$ we have $\|q-p\|\in [t_i,t_i+\varepsilon_d]$, then the distance from $q$ to $S_i^{d-1}$ is at most $K_3\varepsilon^{1/2}_d$ \mbox{for some constant $K_3$.}

For each $q\in P'$, replace $q$ with the closest point on the corresponding $S_i^{d-1}$. Denote the resulting set $P''$. Then the distances between distinct points of $P''$ are contained in
\[\Big[t_{1}-\frac 13\varepsilon_{d-1}, t_{1}+\frac 13\varepsilon_{d-1}\Big]\cup\ldots \cup \Big[t_{k}-\frac 13\varepsilon_{d-1}, t_{k}+\frac 13\varepsilon_{d-1}\Big],\]
provided that $\varepsilon_d$ is small enough. 
Thus, $P''$ is a $\frac 23\varepsilon_{d-1}$-nearly $k$-distance set with distances $1\leq t_1'\leq\dots\leq t_k'$, where the first inequality follows from the assumption that $t_1\geq 2$
and that $\varepsilon_{d-1}$ is sufficiently small.

Since the set $P''$ is obtained by a small perturbation from a subset of $P$, we can show by a simple calculation that there is a constant $K_4$ such that for any $q\in P''$ the set $P''\cup \{p\}$ is $(q,d,\alpha_d+K_4\varepsilon^{1/2}_d)$-flat with respect to a subspace $\Gamma_q$ containing $q-\mathbf{0}$. Let $\Gamma_q'$ be a rotation of $\Gamma_q$ by an angle at most $\alpha_4+K_4\varepsilon^{1/2}_d$ such that $\Gamma_q'$ contains $q-\mathbf{0}$ and $p-\mathbf{0}$. Then, by the triangle inequality, $P''$ is $(q,d,2(\alpha_d+K_4\varepsilon^{1/2}_d))$-flat with respect to $\Gamma_q'$. 

For every $j'<i<j''$ let $M_i$ be the affine $(d-1)$ dimensional plane containing $S_i^{d-1}$, and let $\Delta_i$ be the $(d-1)$-dimensional subspace parallel to $M_i$. If $q\in P''$ is in $S_i^{d-1}$, then $\Gamma'_q$ and $\Delta_i$ are intersecting-orthogonal. Moreover, $\Gamma_q'$ contains the centre of $S^{d-1}_i$. These imply that the set $P''_i:=P''\cap \mathcal{S}_i^{d-1}$ is $(d-1,\varepsilon_{d-1},\mathcal S^{d-1}_i)$-flat, where the subspace for the flatness at point $q$ is $\Gamma'_q\cap \Delta_i$.
Thus, by induction we have $|P''_i|\le 2(k+1)^{d-1}.$

Overall, we have 
\[|P|\le 1+\sum_{i=1}^{j'}|X|+\sum_{i=j''}^k|Y|+\sum_{i=j'+1}^{j''-1}|P''_i|\le 1+k\cdot 2(k+1)^{d-1}\le 2(k+1)^d.\]

A similar, but simpler proof works for non-spherical sets. Let us sketch the proof. We fix a point $p$, decompose the set $P$ into $p$ and the $\varepsilon_d$-neighbourhoods of $k$ spheres at distance $t_i$ from $p$. We then project the points on the corresponding spheres and apply inductive hypothesis for the spherical sets of dimension $d-1$. The only thing to verify is that the sets are $(d-1,\varepsilon_{d-1}, S_i^{d-1})$-flat, and notably that the corresponding approximating plane passes through the centre of the sphere $S_i^{d-1}$. But we may assume that, since the approximating plane $\Gamma_q$ of any point $q\in P\setminus \{p\}$ was containing $\mathbf 0$ and was forming an angle at most $\alpha_d$ with the line $pq$. Thus, by a slight perturbation (and by a possibly slightly weaker approximation), we may assume that both $p$ and $\mathbf 0$ are contained in $\Gamma_q$. Since the centre of $\mathcal S_i^{d-1}$ lies on the line $p\mathbf 0$, it is contained in $\Gamma_q$ as well.

\subsection{Proof of Theorems~\ref{manynearly} and~\ref{thmsmallkexact}}\label{sec6}
Let us start with the proof of the upper bound in Theorem~\ref{thmsmallkexact}. It is immediately implied by the following theorem, combined with the fact that $A_k(d) =m_k(d-1)$ in the cases covered in Theorem~\ref{thmsmallkexact}. (This is what we have actually shown in the proof of Theorem~\ref{nearly k}.)

\begin{thm}\label{thmak} For any $d\ge 2$ and $k\ge 1$  there exists $n_0$, such that for any $n\ge n_0$ we have
\begin{equation}\label{almost}M_k(d,n)\leq T(n,A_k(d))\leq \left (1-\frac{1}{A_k(d)}\right )\frac{n^2}{2}
.\end{equation}
Moreover, \eqref{almost} remains valid if in the definition of $M_k(d,n)$ we change the intervals of the form $[t_i,t_i+1]$ to intervals of the form $[t_i,t_i+cn^{1/d}]$ for some constant $c=c(k,d)$.
\end{thm}

We first prove Theorem~\ref{thmak}, that is, we show  that \eqref{almost} holds with intervals of the form $[t_i,t_i+cn^{1/d}]$, where $c=c(k,d)$ is a sufficiently small constant, to be specified later. The proof of Theorem~\ref{manynearly} is very similar and is actually simpler. We sketch the changes needed to be made in order to prove it in the end of this section.\vskip+0.1cm

Let $\ell=A_k(d)+1$ and let $\alpha,\varepsilon>0$ be fixed such that there exists no almost $(d-1,\alpha)$-flat $\varepsilon$-nearly $k$-distance set in $\mathbb{R}^d$ of cardinality $\ell$.
Assume on the contrary that \eqref{almost} does not hold for some set of $n$ points $S''\subset \R^d$ for sufficiently large $n$. Let $1\leq t_1\le \ldots \le t_k$ be the corresponding distances, and let $c$ be the constant from the statement of the theorem. Our goal is to derive a contradiction by constructing an almost $(d-1,\alpha)$-flat $\varepsilon$-nearly $k$-distance set of cardinality~$\ell$.

In the proof, we will use a hierarchy of ``small'' constants given below. We write $\mu\ll \nu$ if $\mu$ is a certain (positive, but typically  quickly tending to $0$) function, depending on $\nu$ only. Thus, the arrows indicate the order of choosing the parameters: from the right to the left below (and thus, for consistency, one only needs to check that every condition we impose on a constant in the hierarchy only depends on the constants that are to the right from it and is of the form ``it is sufficiently small compared to some of the constants to the right''). Note also that all the constants given below are independent of $n$. 
\begin{equation}\label{hier}1/n\ll c\ll c_1\ll 1/C\ll 1/m\ll 1/M,\delta,c_2, \nu \ll 1/d,1/k,\alpha,\varepsilon.\end{equation}
We recommend the reader to refer to this chain of dependencies throughout the proof.

We use the following simple claim.
\begin{cla} For any $k\ge 0$, we have $N_k(d)<A_{k+1}(d).$
\end{cla}
\begin{proof}
  Take a construction $S$ of a $(d-1,\mu)$-flat $\mu$-nearly $k$-distance set in $\R^d$ in which the distances are at least $K/\mu$ for a sufficiently large $K$. Pick any $p\in S$, and let $\Gamma$ be a subspace of dimension $(d-1)$ such that $S$ is $(p,d-1,\mu)$-flat with respect to $\Gamma$. Let $q\in \mathbb{R}^d$ be a point at distance $1$ apart from $p$ such that $p-q$ is orthogonal to $\Gamma$. 
  Then it is easy to see that $S\cup \{q\}$ is an almost $(d-1,3\mu)$-flat $3\mu$-nearly $(k+1)$-distance set in $\R^d$ if $\mu$ is sufficiently small and $K$ is sufficiently large.
  
  Indeed, for any $r\in S\setminus \{p,q\}$ we have $\big |\|r-p\|-\|r-q\|\big| \leq \mu$ if $K$ is sufficiently large and $\mu$ is sufficiently small. Thus, the only distance between points of $S\cup \{q\}$ that is not $\mu$-close to a distance between points of $S$, is the distance $\|p-q\|$. Then by the triangle inequality we obtain that $S$ is a $3\mu$-nearly $(k+1)$-distance set. Further, the angle between $p-r$ and $q-r$ is at most $\mu$, if $\mu$ is sufficiently small and $K$ is sufficiently large. Thus, again, by the triangle inequality, for any $r\in S\setminus \{p,q\}$ we have that $S$ is $(r,d-1,3\mu)$-flat.
\end{proof}

Using the claim above, we may assume that $t_1\geq c_2 n^{1/d}$. 
Indeed, assume the contrary. Since $S''$ is separated, a volume argument implies that for each vertex $v\in S'',$ the number of vertices in $S''$ at distance at most $c_2 n^{1/d}$ from $v$ is at most $(4c_2)^dn$. Thus, removing all edges from $G''$ that correspond to such distances, we only remove at most $(4c_2)^dn^2$ edges. At the same time, we reduce  the size of the set of possible intervals by at least $1$. Hence, we apply Theorem~\ref{manynearly} with $\nu$ playing the role of $\varepsilon$, and obtain
\[M_k(d,n)\leq (4c_2)^dn^2+M_{k-1}(d,n)\leq (4c_2)^dn^2+\frac{n^2}{2}\left (1-\frac{1}{N_{k-1}(d)}+\nu\right )\leq \frac{n^2}{2}\left (1-\frac{1}{A_{k}(d)}\right )\]
by using the hierarchy \eqref{hier}.

We note here that in the proof of Theorem~\ref{manynearly}, this step is automatic, since the removal of edges corresponding to small distances only change the potential value of $\gamma$.

Our next goal is to obtain a sufficiently structured subset of $S''$. We need the following result of Erd\H os.

\begin{thm}[\cite{Erd69}]\label{erd69} Every $n$-vertex graph  with at least $T(n, \ell-1)+1$ edges contains an edge that is contained in at least $\delta n^{\ell-2}$ cliques of size $\ell$, where $\delta$ is a constant that depends only on $\ell$.
\end{thm}

Consider the graph $G''=(S'',E)$, where the set of edges consist of all pairs of points $\{p_1,p_2\}$ for $p_1,p_2\in S$ that satisfy 
\[\|p_1-p_2\|\in \bigcup_{i=1}^k[t_i,t_i+c n^{1/d}].\]
Using the theorem above, we will show that the following lemma holds.

\begin{lem}\label{completemulti} For any fixed $m$, there exists a choice of $c_1=c_1(m)$  such that $G''$ contains a complete $\ell$-partite subgraph $K_{1,1,m,\ldots,m}$ with the distances between any two of its vertices strictly bigger than $c_1n^{1/d}$.
\end{lem}
\begin{proof}We construct this multipartite graph in three steps. 

\bf Step 1. \rm  Using Theorem~\ref{erd69}, we find an edge $e=\{v_1,v_2\}$ that is contained in at least $\delta n^{\ell-2}$ cliques of size $\ell$. Let $E''$ be the set of those edges of the $\ell$-cliques, that are not incident to $v_1$ or $v_2$. Further, let $F$ be the set of the $(\ell-2)$-tuples formed by the $\ell-2$ vertices of the cliques that are different from $v_1$ and $v_2$. The vertices of $e$ form the first two parts of the multipartite graph. In what follows, we will work with the graph $G''$ induced on $S''\setminus\{v_1,v_2\}$ by $E''$. \vskip+0.1cm

\bf Step 2. \rm We select a set $S_H$ of $C$ vertices of $G''$ at random, and define a hypergraph $H'$ on $S_H$ as follows. Recall that $c_1\ll 1/C\ll \delta,1/\ell, 1/m$ (see \eqref{hier}; the exact dependency of $C$ on $\delta, m$ and of $c_1$ on $C$ shall be clear later), and consider the induced subgraph $G':=G''[S_H]$. $S''$ is separated, hence a volume argument implies that any vertex in $S''\setminus\{v_1,v_2\}$ is at distance strictly bigger than $c_1n^{1/d}$ from all but at most $(4c_1)^dn$ vertices of $S''\setminus\{v_1,v_2\}$. The number of vertices in  $S''\setminus\{v_1,v_2\}$ is  $n-2$, so by the union bound we have the following.
 \begin{quote}
     (i)  With probability at least $1-\binom{C}{2}(4c_1)^d n/(n-2)>1-c_1$, every pair of vertices in $S_H$ is at distance bigger than $c_1n^{1/d}$ from each other.
 \end{quote}
Indeed, the total number of pairs of vertices is $\binom{C}{2}$,  and for each pair the probability that it is at distance $\le c_1n^{1/d}$ is at most $(4c_1)^dn/(n-2)$. The inequality in (i) is possible to satisfy by fixing $\ell,C$ and choosing $c_1$ to be sufficiently small.
 
Next, we consider the $(\ell-2)$-uniform hypergraph $H''=(S''\setminus\{v_1,v_2\},F)$. The following is an easy consequence of a Markov inequality-type argument.
\begin{quote}(ii) With probability at least $\delta/2$, the edge density of the hypergraph \mbox{$H'=H''[S_H]$} is at least $\delta/2$. \end{quote}
Indeed, the average density of cliques should be the same as of $H''$, i.e., at least $\delta$. But if (ii) does not hold, then the average density is at most $(1-\delta/2)\cdot\delta/2+\delta/2\cdot 1=\delta-\delta^2/4<\delta,$ a contradiction.

If we choose $c_1<\delta/2$, then with positive probability both the property in (i) and in (ii) hold. Pick a subset $S_H\subseteq S\setminus \{v_1,v_2\}$  that satisfies both.

\bf Step 3. \rm We apply the following hypergraph generalisation of the K\H ov\'ari--S\'os--Tur\'an theorem due to Erd\H os.

\begin{thm}[\cite{Erd64}]\label{Erd}
For any $\ell\geq 4$, $m\geq 1$, $\delta>0$ there is a constant $C(\ell,m,\delta)$ such that the following holds for any $C\geq C(\ell,m,\delta)$. Any $(\ell-2)$-uniform hypergraph on $C$ vertices of edge density at least $\frac{\delta}{2}$
contains a copy of a complete $(\ell-2)$-partite $(\ell-2)$-uniform hypergraph with parts of size $m$. 
\end{thm}
 Applying the theorem to the $(\ell-2)$-hypergraph $H'$, we obtain a complete $(\ell-2)$-partite $(\ell-2)$-uniform hypergraph with parts of size $m$. This complete multipartite hypergraph corresponds to a complete $(\ell-2)$-partite graph in $G$ with parts of size $m$ and with all distances between points being at least $c_1n^{1/d}.$ Together with the edge $e$, this gives the desired $\ell$-partite subgraph $K_{1,1,m,\dots,m}$.
\end{proof}

Let the $\ell$ parts of the $K_{1,1,m\dots,m}$ in $G''$ be $S'_1,\dots,S'_{\ell}$, with $S_1=\{v_1\}$, $S_2=\{v_2\}$ and with $|S_3|=\dots=|S_{\ell}|=m$, further set $S'=S_1\cup\dots\cup S_{\ell}$. $S'$ has much more structure than the original set $S''$. However, distances from several intervals from $[t_1,t_1+cn^{1/d}],\dots,[t_k,t_k+cn^{1/d}]$ may appear between the vertices of $S'_i$ and $S'_j$ ($i\ne j$). To reduce it to one interval between any two parts, we will do the second ``preprocessing'' step using the following version of the K\H ov\'ari--S\'os--Tur\'an theorem.
\begin{thm}[\cite{KST}]\label{thmkst} For any $\zeta>0$ and $r\ge 1$  there exists $n_0$, such that for any $n\ge n_0$ we have the following. Any graph on $n$ vertices with at least $\zeta\binom{n}{2}$ edges contains $K_{r,r}$ as a subgraph.
\end{thm}

Take $S'$ and set $i:=1$. Then do the following procedure.
\begin{itemize}
\item[1.] Set $j:=i+1$. If $i=1,j=2$, set $j:=3$.
\item[2.] Take the subgraph of $G'$ induced between $S_i'$ and $S_j'$. 
Choose an index $\psi=\psi(i,j)\in[k]$, such that
\[\bigl\lvert \bigl\{(v_i,v_j):v_i\in S_i', v_j\in S_j', |v_i-v_j|\in [t_\psi,t_\psi+c n^{1/d}]\bigr\} \bigr\rvert\ge \frac {m^{\sigma}}k,\]
where $\sigma=1$ if $i\in\{1,2\}$ and $\sigma=2$ otherwise.
Set $G_{ij}$ be the graph between $S_i'$ and $S_j'$ with the set of edges specified in the displayed formula above.
\item [3.] If $i\in\{1,2\}$, let $S''_i$ be the set of neighbours of $p_i$ in $G_{ij}$. If $i\notin\{1,2\}$, apply Theorem~\ref{thmkst} to $G_{ij}$ and find sets $S''_i\subset S'_i$, $S''_j\subset S'_j$, each of size $1\ll m'\ll m$, such that the graph $G_{ij}$ between $S''_i$ and $S''_j$ is complete bipartite.
\item[5.] Set $S'_i:=S''_i$, $S'_j:=S''_j$, $m:=m'$, $j:=j+1$. If $j\le k$ then go to Step 2. If $j>k$ then set $i:=i+1$. If $i\ge k$, then terminate, otherwise go to Step $1$.
\end{itemize}

Clearly, if $m$ in the beginning of the procedure was large enough, then at the end of the procedure $m$ is still larger than some sufficiently large $M$.
By running a procedure similar to the one above, we can shrink the parts $S_i$'s further such that for any $p_i\in S_i$ and $p_j,q_j\in S_j$ $(j\notin\{1,2\})$ the angle $\angle p_jp_iq_j$ is at most $\alpha.$ If $M$ is sufficiently large (see the hierarchy \eqref{hier}), then at the end of this second procedure each $S_i$ ($i\notin\{1,2\}$) has at least $2$ points. Thus, we obtain a subset $S\subset S'$, such that $G:=G''[S]$ is complete multipartite with parts $S_1,\ldots, S_{\ell}$ such that $|S_1|=|S_2|=1$ and $|S_3|=\dots=|S_{\ell}|=2$, moreover for any two parts $S_i, S_j$ there is an index $\psi(i,j)\in [k]$ such that \begin{multline}\label{eq03} \text{for any }p_i\in S_i,\ p_j,q_j\in S_j \text{ we have }\|p_i-p_j\|\in [t_{\psi(i,j)},t_{\psi(i,j)}+c n^{1/d}]\\\text { and } \angle p_jp_iq_j\leq \alpha.\end{multline} 

For each $3\leq i \leq \ell$ let $S_i=\{p_i,q_i\}$. Let $P$ be the set $\{p_1,\dots,p_{\ell}\}$ scaled by $\frac{1}{c_2n^{1/d}}$, that is, let $P=\frac{1}{c_2n^{1/d}}\{p_1,\dots,p_{\ell}\}$.  We will show that $P$ is an almost $(d-1,\alpha)$-flat $\varepsilon$-nearly $k$-distance set, and obtain the desired contradiction. Indeed, this set is separated, since all the distances between $p_i$ and $p_j$ for $i\neq j$ were at least $c_2n^{1/d}$. Further, it is an $\varepsilon$-nearly $k$-distance set, since the length of each of the intervals in which the distances fall is  $cn^{1/d}/c_2n^{1/d}=c/c_2\le \varepsilon$.

Finally, we claim that for any $i\notin\{1,2\}$ and any $j\neq i$ we have $\angle q_ip_ip_j\in [\frac{\pi}{2}-\alpha,\frac{\pi}{2}+\alpha]$. Let us show this. Take the point $q_i'$ on the line through $p_i,p_j$ such that $\|q_i-p_j\|=\|q_i'-p_j\|.$ Then, first, $\angle q_iq_i'p_j\in[(\pi-\alpha)/2,\pi/2]$ since $\angle q_ip_jp_i\le \alpha$ and the triangle $q_iq_i'p_j$ is isosceles. Second, we have $\|q_i'-p_i\|\le cn^{1/d}.$
Since $\|q_i-p_i\|\ge c_1n^{1/d}$, we may assume that $\angle q_i'q_ip_i\le \alpha/2$, and thus $\angle q_ip_ip_j\in [(\pi-\alpha)/2-\angle q_i'q_ip_i,\pi/2+\angle q_i'q_ip_i]\subset [\pi/2-\alpha,\pi/2+\alpha]$.

Thus, for every $i\notin \{1,2\}$ we have that $P$ is $(\frac{1}{c_2n^{1/d}}p_i,d-1,\alpha)$-flat with respect to the $d-1$ dimensional subspace orthogonal to $(p_i-q_i)$. This finishes the proof Theorem~\ref{thmak}.

\medskip

We now turn to the proof of Theorem~\ref{manynearly}. We prove that for every $\gamma>0$ inequality \eqref{eq01} holds with intervals of the form $[t_i,t_i+cn^{1/d}]$ where $c=c(k,d,\gamma)$ if $n$ is sufficiently large. As the proof is very similar to the proof of Theorem~\ref{thmak}, we only sketch it, pointing out the differences.

Let $\ell:=N_k(d)+1$ and $\alpha,\varepsilon>0$ be fixed such that there exists no $(d-1,\alpha)$-flat $\varepsilon$-nearly $k$-distance set in $\mathbb{R}^d$ of cardinality $\ell$.
Assume on the contrary that for every $c>0$ and $n_0$ there is an $n\geq n_0$, there are $k$ distances $t_1<\dots\leq t_k$ and a set $S''\subset \R^d$ of $n$ points for which
\[\Big|\setbuilder{(p,q)\in S''\times S''}{\|p-q\|\in [t_i,t_i+cn^{1/d}] \textrm{ for some } i\in [k]}\Big|>T(N_k(d),n)+\gamma n^2.\]
Our goal is to derive a contradiction by constructing an a $(d-1,\alpha)$-flat $\varepsilon$-nearly $k$-distance set of cardinality~$\ell$.

After including $\gamma$ in the hierarchy of constants on the same level as $\alpha$, the proof is the same as that of \eqref{thmak} up to the point of Lemma~\ref{completemulti}. Instead of Lemma~\ref{completemulti} we will use the following.

\begin{lem}\label{completemulti2}For any fixed $m$, there exists a choice $c_1=c_1(m,\gamma)$ such that $G''$ contains a complete $\ell$-partite subgraph $K_{m,\dots,m}$ such that the distance between any two of its vertices is bigger than $c_1n^{1/d}$.
\end{lem}

The proof of Lemma~\ref{completemulti2} is very similar to the proof of Lemma~\ref{completemulti}, except that instead of Theorem~\ref{erd69} we use a result of Erdős and Simonovits \cite{ESsat} about the supersaturation of $\ell$-cliques. (And then work with $\ell$-uniform hypergraphs instead of $\ell-2$.) Therefore, we only give an outline of the proof.

\begin{thm}[\cite{ESsat}]\label{LS} For any $\ell,\gamma>0$ there is a $\delta$ such that if a graph $G$ on $n$ vertices has at least $T(n,\ell)+\gamma n^2$ edges, then it contains at least $\delta n^{\ell}$ cliques of size $\ell$.
\end{thm}

\begin{proof}[Sketch of proof of Lemma~\ref{completemulti2}]
We construct this multipartite graph in three steps.

\bf Step 1. \rm  Using Theorem~\ref{LS}, we find $\delta n^{\ell}$ cliques of size $\ell$. Let $E''$ be the set of the $\ell$-cliques, and $F$ be the set of the $\ell$-tuples. In what follows, we will work with the graph $G''$ induced on $S''$ by $E''$. \vskip+0.1cm

 \bf Step 2. \rm Select $C$ vertices of $G''$ at random, where $c_1\ll 1/C\ll \delta,1/\ell, 1/m$. Denote by $S_H$ the set of $C$ vertices that we chose and consider the induced subgraph $G':=G''[S_H]$. A similar calculation as in the proof of Lemma~\ref{completemulti} implies the following.
 \begin{quote}
     (i)  With probability at least $>1-c_1$, every pair of vertices in $S_H$ is at distance bigger than $c_1n^{1/d}$ from each other.
 \end{quote}
 
Next, we consider the $\ell$-uniform hypergraph $H''=(S'',F)$. As before we obtain the following.
\begin{quote}(ii) With probability at least $\delta/2$, the edge density of the hypergraph $H'=H''[S_H]$ is at least $\delta/2$. \end{quote}

If we choose $c_1<\delta/2$ then with positive probability both the property in (i) and in (ii) hold. Pick a subset $S_H\subseteq S$  that satisfies both.

\bf Step 3. \rm Applying Theorem~\ref{Erd} to the $\ell$-hypergraph $H'$, we obtain a complete $\ell$-partite $\ell$-uniform hypergraph with parts of size $m$. This complete multipartite hypergraph corresponds to a complete $\ell$-partite graph in $G$ with parts of size $m$ and with all distances between points being at least $c_1n^{1/d}.$ 
\end{proof}

Let the $\ell$ parts of the $K_{m\dots,m}$ in $G''$ be $S'_1,\dots,S'_{\ell}$, with $|S_1|=\dots=|S_{\ell}|=m$ and set $S'=S_1\cup\dots\cup S_\ell$. Running a similar procedure as before we obtain a subset $S\subset S'$, such that $G:=G''[S]$ is complete multipartite with parts $S_1,\ldots, S_{\ell}$ such that $|S_1|=\dots=|S_{\ell}|=2$, moreover for any two parts $S_i, S_j$ there is an $\psi(i,j)\in [k]$ with \begin{multline*}\label{eq03} \text{for any }p_i\in S_i,\ p_j,q_j\in S_j \text{ we have }\|p_i-p_j\|\in [t_{\psi(i,j)},t_{\psi(i,j)}+c n^{1/d}]\\\text { and } \angle p_jp_iq_j\leq \alpha.\end{multline*} 
For each $1\leq i \leq \ell$ let $S_i=\{p_i,q_i\}$. Then we can show that $P=\frac{1}{c_2n^{1/d}}\{p_1,\dots,p_{\ell}\}$ is a $(d-1,\alpha)$-flat $\varepsilon$-nearly $k$-distance set, and obtain a contradiction.

\section{Concluding remarks}

Let us list some of the intriguing open problems that arose in our studies. One important step forward would be to get rid of the (almost-)flatness in the relationship between nearly $k$-distance sets and the quantity $M_k(d,n)$ that appears in Theorems~\ref{manynearly} and~\ref{thmak}. In particular, it would be desirable to prove the first equality in Conjecture~\ref{conj1} and, more generally, show the following.
\begin{prb}
Show that $A_k(d+1,d) = N_k(d+1) = M_k(d)$ holds for any $k,d.$
\end{prb}
In fact, even showing the first equality would imply that the value of $M_k(d,n)$ for large $n$ is determined {\it exactly} by the value of $N_k(d+1)$.

Another interesting question that looks approachable is to determine the value of $M_k(d)$ on the part of the spectrum opposite to that of Theorem~\ref{nearly k}: for any fixed $d$ and $k$ sufficiently large. Note
that the order of magnitude of $M_k(d)$ in this regime is easy to find, as it is shown in Theorem~\ref{fixd}.
\begin{prb}
Determine $M_k(d)$ for any fixed $d$ and sufficiently large $k.$
\end{prb}
If resolved, then with some effort it would most likely be possible to determine the value of $M_k(d,n)$ for large $n$ in this regime as well. 

\section*{Acknowledgements}
We thank Konrad Swanepoel for introducing us to the topic, for many fruitful discussions, and for his helpful comments on the manuscript. We are also grateful for Peter Allen and the anonymous referees for suggestions on improving the presentation of the paper. This research was done while the first author was a PhD student at the London School of Economics.

\section*{Appendix}

\begin{proof}[Proof of Claim~\ref{cla1}]
  First, assume that $a_1<z_1/2$. Then \begin{small}$${z_1+z_2-1\choose a_1+a_2}> {z_1\choose a_1}{z_2-1\choose a_2}+{z_1\choose a_1+1}{z_2-1\choose a_2-1}\ge {z_1\choose a_1}\Big({z_2-1\choose a_2}+{z_2-1\choose a_2-1}\Big) = {z_1\choose a_1}{z_2\choose a_2}.$$\end{small}
The proof is the same for $a_2<z_2/2$. Finally, assume that $a_1=z_1/2$ and $a_2=z_2/2$ (and thus that $a_1,a_2\ge 2$). Since $a_1=a_2=2$ is excluded, assume that $a_1\ge 3$. We use the following inequalities: 
\begin{small}\[{z_1\choose a_1+1}=\frac {z_1-a_1}{a_1+1}{z_1\choose a_1} \ge \frac 34{z_1\choose a_1}\]
\end{small}
and \begin{small}
\[{z_1\choose a_1-1}{z_2-1\choose a_2+1}=\frac {a_1}{z_1-a_1+1}{z_1\choose a_1}\frac {z_2-a_2-1}{a_2+1}{z_2-1\choose a_2}\ge \frac 14{z_1\choose a_1}{z_2-1\choose a_2}=\frac 14{z_1\choose a_1}{z_2-1\choose a_2-1}.\]
\end{small}
Using these two inequalities, we can repeat the calculations as above:
\begin{multline*}
  {z_1+z_2-1\choose a_1+a_2}\ge {z_1\choose a_1}{z_2-1\choose a_2}+{z_1\choose a_1+1}{z_2-1\choose a_2-1}+{z_1\choose a_1-1}{z_2-1\choose a_2+1}\\
  \ge {z_1\choose a_1}{z_2-1\choose a_2}+{z_1\choose a_1}{z_2-1\choose a_2-1}={z_1\choose a_1}{z_2\choose a_2}.
\end{multline*}
\end{proof}

\begin{proof}[Proof of \eqref{eqcheck}]
Using the known values of $m_2(d)$ and bounds on $m_3(d),$ we obtain the following.
\begin{itemize}
\item[$d=8$:] $\max\{(j+1)m_2(d-j): j=0,\ldots, 8\} = \max\{45, 2\cdot 29, 3\cdot 27, 4\cdot 16, 5\cdot 10, 6\cdot 6, 7\cdot 5, 8\cdot 3, 9\cdot 1\} = 81\leq 121 \leq m_3(8)$; 

\item[$d=7$:] $\max\{(j+1)m_2(d-j): j=0,\ldots, 7\}=\max\{29, 2\cdot 27,3\cdot 16, 4\cdot 10,5\cdot 6, 6\cdot 5, 7\cdot 3,8\cdot 1\}=
54\leq 65 \leq m_3(7)$;

\item[$d=6$:] $\max\{(j+1)m_2(d-j): j=0,\ldots, 6\}=\max\{27,2\cdot 16,3\cdot 10,4\cdot 6,5\cdot 5,6\cdot 3,7\cdot 1\}=32\leq 40\leq m_3(6)$;

\item[$d=5$:] $\max\{(j+1)m_2(d-j): j=0,\ldots, 5\}=\max\{16,2\cdot 10,3\cdot 6,  4\cdot 6, 5\cdot 3, 6\cdot 1\}=24\leq m_3(5)$;

\item[$d=4$:] $\max\{(j+1)m_2(d-j): j=0,\ldots, 4\}=\max\{10,2\cdot 6,3\cdot 5, 4\cdot 3, 5\cdot 1\}=15\leq16= m_3(4)$;

\item[$d=3$:] $\max\{(j+1)m_2(d-j): j=0,\ldots, 3\}=\max\{6,2\cdot 5,3\cdot 3, 4\cdot 1\}=10\leq 12 = m_3(3)$;

\item[$d=2$:] $\max\{(j+1)m_2(d-j): j=0,1, 2\}=\max\{5,2\cdot 3,3\cdot 1\}=6\leq7= m_3(2)$;

\item[$d=1$:] $\max\{(j+1)m_2(d-j): j=0,1\}=\max\{3,2\cdot 1\}=3\leq4=  m_3(1)$.

\end{itemize}
\end{proof}

\begin{proof}[Proof of Claim~\ref{beta}]
Let $q'$ be a the translate of $p$ by $v_q$, and $r'$ be the translate of $r$ by $r'$. Then $r'-q'$ is parallel $v_q-v_r$. Let $\beta_1=\angle qrq'$ and $\beta_2=\angle rq'r'$. Then the angle between $q'-r'$ and $q-r$ is at most $\beta_1+\beta_2$, thus it is sufficient to show that $\beta_1,\beta_2\leq 10 (K\alpha)^{1/2}$. We will prove it for $\beta_2$, for $\beta_1$ it can be done similarly.
By the low of cosines we have \[\cos \beta_2=\frac{\|q'-r\|^2+\|q'-r'\|^2-\|r-r^2\|}{2\|q'-r\|\|q'-r'\|}.\]
By the triangle inequality we have \[\|q-r\|-\|q-q'\| \leq \|q'-r\|\leq \|q-r\|+\|q-q'\|\] and
\[\|q-r\|-\|q-q'\|-\|r-r'\| \leq \|q'-r'\|\leq \|q-r\|+\|q-q'\|+\|r-r'\|.\] Further, we have
\[\|q-q'\|=2\sin \alpha \|p-q\|\leq 2\alpha K \|q-r\|\] and \[\|r-r'\|=2\sin \alpha \|p-r\|\leq 2\alpha K \|q-r\|,\] where in both cases the inequality follows by $\sin\alpha\leq \alpha$, and by the assumption that $\frac{\|p-q\|}{\|q-r\|}\leq K$.
By denoting $\|q-r\|=z$, the inequalities above imply
\[1-\cos \beta_2\leq 1-\frac{2(z-4\alpha K z)^2-4(\alpha K z)^2}{2(z+4\alpha K z)^2}\leq 25 \alpha K.\]
Combining this with $\frac{\beta_2^2}{4}\leq 1-\cos \beta_2$ we obtain $\beta_2\leq 10(\alpha K)^{1/2}$.
\end{proof}

\end{document}